\documentclass[letterpaper,11pt]{article}
\usepackage{amsmath}
\usepackage{amssymb}
\usepackage{geometry}

\newtheorem{theorem}{Theorem}

\newtheorem{corollary}[theorem]{Corollary}

\newtheorem{definition}[theorem]{Definition}
\newtheorem{example}[theorem]{Example}

\newtheorem{lemma}[theorem]{Lemma}

\newtheorem{proposition}[theorem]{Proposition}
\newtheorem{remark}[theorem]{Remark}

\newenvironment{proof}[1][Proof]{\noindent\textbf{#1.} }{\ \rule{0.5em}{0.5em}}
\begin{document}

\title{Limit Behaviour of Sequential Empirical Measure Processes}
\author{Omar El-Dakkak \\
Laboratoire de Statistique Th\'{e}orique et Appliqu\'{e}e (L.S.T.A.)\\
Universit\'{e} Paris VI}
\date{}
\maketitle

\begin{abstract}
In this paper we obtain some uniform laws of large numbers and functional
central limit theorems for sequential empirical measure processes indexed by
classes of product functions satisfying appropriate Vapnik-\v{C}ervonenkis
properties.


\end{abstract}

\section{Introduction}

\noindent Let $\left( X_{n}\right) _{n\geq 1}$ be a sequence of i.i.d.
random variables of law $\nu ,$ defined on a probability space $\left(
\Omega ,\mathcal{A},P\right) $ and taking values in some measurable space $%
\left( U,\mathcal{U}\right) $. Let $\mathcal{Q}$ be a class of measurable
functions%
\begin{equation*}
q:\left( \left[ 0,1\right] \times U,\mathcal{B}\left( \left[ 0,1\right]
\right) \otimes \mathcal{U}\right) \rightarrow \left( \mathfrak{R},\mathcal{B%
}\left( \mathfrak{R}\right) \right) ,
\end{equation*}%
where $\mathcal{B}\left( \left[ 0,1\right] \right) $ and $\mathcal{B}\left(
\mathfrak{R}\right) $ denote the class of Borel sets of $\left[ 0,1\right] $
and $\mathfrak{R},$ respectively. The stochastic process%
\begin{equation}
\mathbb{P}_{n}\left( q\right) \triangleq \frac{1}{n}\sum_{i=1}^{n}q\left(
i/n,X_{i}\right) =\frac{1}{n}\sum_{i=1}^{n}\delta _{\left( i/n,X_{i}\right)
}\left( q\right) ,\qquad q\in \mathcal{Q},  \label{iniziale}
\end{equation}%
is called \textit{the sequential empirical measure process (s.e.m.p.)
indexed by }$\mathcal{Q}.$

\bigskip

\noindent It is known that any symmetric statistic can be seen as a
functional of the classical empirical measure. The fact that the sequential
empirical measure%
\begin{equation*}
\mathbb{P}_{n}\triangleq \frac{1}{n}\sum_{i=1}^{n}\delta _{\left(
i/n,X_{i}\right) }
\end{equation*}%
enables to reconstruct the whole sequence $X_{1},...,X_{n}$ (rather than the
sample up to a permutation) makes of $\mathbb{P}_{n}$ a very flexible tool
to represent complex and highly non-symmetric statistics. One can, for
instance think of two-sample sequential rank statistics, $V$-statistics or
fractional ARIMA processes as treated in Barbe and Broniatowski $\left[ 1997%
\right] ,$ $\left[ 1998a\right] $ and $\left[ 1998b\right] $. In these
references, sequential empirical measure representations allowed to obtain
functional large deviation principles for classes of complex statistics as
the ones evoked. Weak convergence in the case in which $\mathcal{Q}=\left\{
\mathbf{1}_{\left[ 0,t\right] }\cdot f:0<t\leq 1,\text{ }f\in \mathcal{F}%
\right\} ,$ where $\mathcal{F}$ is a Donsker class, has been treated in van
der Vaart and Wellner $\left[ 1996\right] .$ In that case, it was found that
weak convergence occured in $\ell ^{\infty }\left( \left[ 0,1\right] \times
\mathcal{F}\right) ,$ the limiting process being the standard Kiefer-M\"{u}%
ller process indexed by $\left[ 0,1\right] \times \mathcal{F}.$

\bigskip

\noindent The aim of this paper is to obtain uniform laws of large numbers
(U.L.L.N's) and functional central limit theorems (F.C.L.T's) for sequential
empirical measure processes indexed by classes of product functions
satisfying certain properties. Our results generalize those in van der Vaart
and Wellner $\left[ 1996\right] $ and allow to obtain uniform results for
classes of non-symmetric statistics of independent random variables whose
laws depend non linearly on time. More precisely, let $\mathcal{H}\subset
\mathfrak{R}^{\left[ 0,1\right] }$ and $\mathcal{G}\subset \mathfrak{R}^{U}$
be two classes of functions. Define the class%
\begin{equation*}
\mathcal{F}_{\pi }\triangleq \mathcal{H\cdot G}\triangleq \left\{ f=hg:h\in
\mathcal{H},\text{ }g\in \mathcal{G}\right\} .
\end{equation*}

\noindent In the sequel, we will consider combinations of the following
three conditions:

\begin{enumerate}
\item[(H1)] $\mathcal{H}$ \textit{i}s a uniformly bounded Vapnik-\v{C}%
ervonenkis graph class (V.C.G.C.) of almost everywhere continuous functions;

\item[(H2)] $\mathcal{G}$ is a V.C.G.C. with envelope $G\in L_{1}\left( \nu
\right) $ such that, for all $x\in U,$ $G\left( x\right) <+\infty .$

\item[(H3)] $\mathcal{G}$ is a V.C.G.C. with envelope $G\in L_{2}\left( \nu
\right) $ such that, for all $x\in U,$ $G\left( x\right) <+\infty .$
\end{enumerate}

\noindent For an extensive account on Vapnik-\v{C}ervonenkis Theory, the
reader is referred to Chapter 4 of Dudley $\left[ 1999\right] .$ To avoid
any measurability issues, all function classes considered in this paper will
be assumed to be \textit{permissible} in the sense of Appendix C of Pollard $%
\left[ 1984\right] .$ We will now introduce some notation to be used
throughout the paper. First note that, since $\nu $ is a finite measure,
condition (H3) implies condition (H2) but the converse fails in general. We
will write $\mathcal{F}_{\pi }\in \pi \left( \nu G,J\text{-}VC\right) $ if
conditions (H1) and (H2) are verified and if $\mathcal{F}_{\pi }$ is jointly
a V.C.G.C. (i.e. as a function of two variables). We will write $\mathcal{F}%
_{\pi }\in \pi \left( \nu G^{2},J\text{-}VC\right) $ if conditions (H1) and
(H3) are verified and if $\mathcal{F}_{\pi }$ (as a function of two
variables) is a V.C.G.C.. Finally, we will write $\mathcal{F}_{\pi }\in
\mathcal{\pi }\left( UB,M\text{-}VC\right) $ if condition (H1) is verified
and condition (H2) (or, equivalently, condition (H3)) is verified with $G$ a
constant function, i.e. $\mathcal{G}$ is uniformly bounded as well. In this
respect, it is to be noticed that V.C. properties of two function classes
are not inherited by the class of their pointwise products. Similarly, a
V.C.G.C. of functions of two variables does not transmit its V.C. properties
to the two classes of one variable obtained by marginalization or
projection. For more details, the reader is referred to Dudley $\left[ 1999%
\right] $ and Stelgle ad Yukich $\left[ 1989\right] .$ In the sequel,
whenever a uniformly bounded class of function is considered, it will be
tacitly assumed, without loss of generality, that the envelope is
identically equal to $1.$

\noindent Since condition (H1) ensures Riemann integrability of every $h\in
\mathcal{H}$ (see Ash $\left[ 1972\right] ,$ Theorem 1.7.1), we have that,
for all $h\in \mathcal{H},$%
\begin{equation*}
\lim_{n\rightarrow +\infty }\lambda _{n}\left( h\right) \triangleq \int_{
\left[ 0,1\right] }h\left( s\right) \lambda _{n}\left( \text{d}s\right) =%
\frac{1}{n}\sum_{i=1}^{n}h\left( i/n\right) =\int_{\left[ 0,1\right]
}h\left( s\right) \lambda \left( \text{d}s\right) \triangleq \lambda \left(
h\right) ,
\end{equation*}%
where%
\begin{equation*}
\lambda _{n}\triangleq n^{-1}\sum_{i=1}^{n}\delta _{i/n},
\end{equation*}%
will be referred to as the \textit{uniform discrete measure on }$\left[ 0,1%
\right] $ and where $\lambda $ denotes the trace of Lebsegue measure on $%
\left[ 0,1\right] .$

\noindent The paper is organized as follows: in Section \ref{ulln's} the
main results uniform laws of large numbers are presented and discussed. The
same is done in Section \ref{fclt's} for functional central limit theorems.
All the proofs are contained in the last section.

\section{Uniform laws of large numbers\label{ulln's}}

We start with some classical results. Let $\left( \Upsilon ,\mathcal{E}%
\right) $ be a measurable space and let $\mathcal{F}$ be a class of $%
\mathcal{E}$-measurable functions $f:\left( \Upsilon ,\mathcal{E}\right)
\rightarrow \left( \mathfrak{R},\mathcal{B}\left( \mathfrak{R}\right)
\right) .$ Let $F$ be an $\mathcal{E}$-measurable envelope of $\mathcal{F}.$
Let $\left( w_{ni}\right) _{n\geq 1,1\leq i\leq i\left( n\right) },$ with $%
\lim_{n}i\left( n\right) =+\infty ,$ be a triangular array of random
probability measures on $\left( \Upsilon ,\mathcal{E}\right) ,$ and let $%
\left( \xi _{ni}\right) _{n\geq 1,1\leq i\leq i\left( n\right) }$ be a
triangular array of real-valued random variables defined on $\left( \Upsilon
,\mathcal{E}\right) $. The $\mathcal{F}$-indexed stochastic process%
\begin{equation}
S_{n}\left( f\right) \triangleq \sum_{i\leq i\left( n\right) }w_{ni}\left(
f\right) \xi _{ni},  \label{rmp}
\end{equation}%
is called \textit{random measure process.} General random measure processes
have been introduced and studied in Gaenssler $\left[ 1992,\text{ }1993%
\right] ,$ Gaenssler and Ziegler $\left[ 1994b\right] ,$ Gaenssler, Rost and
Ziegler $\left[ 1998\right] $ and Ziegler $\left[ 1997\right] .$ This
treatment allowed to obtain uniform laws of large numbers and functional
central limit theorems for smoothed empirical processes, partial-sum
processes with fixed and random locations, empirical versions of the non
parametric regression function and empirical-type estimators of the
intensity measure of spatial Poisson processes.

\bigskip

To see that sequential empirical measure processes are special cases of
random measure processes, it suffices to take $\Upsilon =\left[ 0,1\right]
\times U,$ $\mathcal{E}=\mathcal{B}\left( \left[ 0,1\right] \right) \otimes
\mathcal{U},$ $i\left( n\right) =n,$ $\xi _{ni}\equiv n^{-1}$ and $%
w_{ni}=\delta _{\left( i/n,X_{i}\right) }.$ In Gaenssler, Rost and Ziegler $%
\left[ 1998\right] ,$ the following result is shown:

\begin{theorem}
\label{gaenssler-rost-ziegler98}\textbf{(Gaenssler, Rost and Ziegler 1998). }%
Assume the following conditions hold:

\begin{enumerate}
\item[(i)] There exists $p\geq 1$ such that, for all $\delta >0,$%
\begin{equation*}
\lim_{n\rightarrow +\infty }\sum_{1\leq i\leq i\left( n\right) }E^{\frac{1}{p%
}}\left( w_{ni}\left( F\right) ^{p}\cdot \left\vert \xi _{ni}\right\vert
^{p}\cdot \mathbf{1}_{\left( \delta ,+\infty \right) }\left( w_{ni}\left(
F\right) \left\vert \xi _{ni}\right\vert \right) \right) =0;
\end{equation*}

\item[(ii)] There exists $\delta _{1}>0$ such that%
\begin{equation*}
\sup_{n\geq 1}\sum_{1\leq i\leq i\left( n\right) }E\left( w_{ni}\left(
F\right) \cdot \left\vert \xi _{ni}\right\vert \cdot \mathbf{1}_{\left[
0,\delta _{1}\right] }\left( w_{ni}\left( F\right) \left\vert \xi
_{ni}\right\vert \right) \right) <+\infty ;
\end{equation*}

\item[(iii)] For all $\tau >0,$ there exists $\delta =\delta \left( \tau
\right) $ such that the sequence%
\begin{equation*}
\left\{ N\left( \tau \mu _{n\delta }\left( F\right) ,\mathcal{F},d_{\mu
_{n\delta }}^{\left( 1\right) }\right) :n\geq 1\right\} ,
\end{equation*}%
of random covering numbers is stochastically bounded, where the random
measure $\mu _{n\delta }$ is defined by%
\begin{equation*}
\mu _{n\delta }\triangleq \sum_{1\leq i\leq i\left( n\right) }w_{ni}\cdot
\left\vert \xi _{ni}\right\vert \cdot \mathbf{1}_{\left[ 0,\delta \right]
}\left( w_{ni}\left( F\right) \left\vert \xi _{ni}\right\vert \right) ,
\end{equation*}%
and the random pseudo-metric $d_{\mu _{n\delta }}^{\left( 1\right) }$ is
defined on $\mathcal{F}$ by%
\begin{equation*}
d_{\mu _{n\delta }}^{\left( 1\right) }\left( f,g\right) \triangleq
\sum_{1\leq i\leq i\left( n\right) }\left\vert w_{ni}\left( f\right)
-w_{ni}\left( g\right) \right\vert \cdot \left\vert \xi _{ni}\right\vert
\cdot \mathbf{1}_{\left[ 0,\delta \right] }\left( w_{ni}\left( F\right)
\left\vert \xi _{ni}\right\vert \right) ,\quad f,g\in \mathcal{F}.
\end{equation*}
\end{enumerate}

Then, the sequence $\left\{ S_{n}\left( f\right) :f\in \mathcal{F}\right\} ,$
$n\geq 1,$ of random measure processes as defined in (\ref{rmp}) verifies
the following version of the $L_{p}$-uniform law of large numbers:%
\begin{equation*}
\sup_{f\in \mathcal{F}}\left\vert S_{n}\left( f\right) -E\left( S_{n}\left(
f\right) \right) \right\vert \overset{L_{p}}{\rightarrow }0.
\end{equation*}
\end{theorem}

\noindent The reader is referred to Section \ref{ancillary} for the
definition of covering numbers. As a consequence of Theorem \ref%
{gaenssler-rost-ziegler98}, we have the following U.L.L.N. for $\mathcal{F}%
_{\pi }$-indexed sequential empirical measure processes

\begin{corollary}
\label{corollary-ulln}If $\mathcal{F}_{\pi }\in \pi \left( \nu G^{1},J\text{-%
}VC\right) ,$ then%
\begin{equation}
\sup_{f\in \mathcal{F}_{\pi }}\left\vert \mathbb{P}_{n}\left( f\right)
-\left( \lambda _{n}\otimes \nu \right) \left( f\right) \right\vert \overset{%
L_{1}}{\rightarrow }0.  \label{ulln(nugi)-lambdaenne}
\end{equation}%
If, moreover, the class $\mathcal{H}$ in the definition of $\mathcal{F}_{\pi
}$ is such that%
\begin{equation}
\sup_{h\in \mathcal{H}}\left\vert \lambda _{n}\left( h\right) -\lambda
\left( h\right) \right\vert \rightarrow 0,  \label{unif-riem-integ}
\end{equation}%
then%
\begin{equation}
\sup_{f\in \mathcal{F}_{\pi }}\left\vert \mathbb{P}_{n}\left( f\right)
-\left( \lambda \otimes \nu \right) \left( f\right) \right\vert \overset{%
L_{1}}{\rightarrow }0.  \label{ulln(nugi)}
\end{equation}%
On the other hand, if $\mathcal{F}_{\pi }\in \pi \left( UB,M\text{-}%
VC\right) ,$ then, for all $p\geq 1,$%
\begin{equation}
\sup_{f\in \mathcal{F}_{\pi }}\left\vert \mathbb{P}_{n}\left( f\right)
-\left( \lambda _{n}\otimes \nu \right) \left( f\right) \right\vert \overset{%
L_{p}}{\rightarrow }0.  \label{ulln(ub)-lambdaenne}
\end{equation}%
If, moreover, condition (\ref{unif-riem-integ}) holds, then, for all $p\geq
1,$%
\begin{equation}
\sup_{f\in \mathcal{F}_{\pi }}\left\vert \mathbb{P}_{n}\left( f\right)
-\left( \lambda \otimes \nu \right) \left( f\right) \right\vert \overset{%
L_{p}}{\rightarrow }0.  \label{ulln(ub)}
\end{equation}
\end{corollary}

\begin{proof}
If $\mathcal{F}_{\pi }\in \pi \left( UB,M\text{-}VC\right) $ or $\mathcal{F}%
_{\pi }\in \pi \left( \nu G^{1},J\text{-}VC\right) ,$ and taking $f=hg,$ then%
\begin{eqnarray*}
E\left( \mathbb{P}_{n}\left( f\right) \right)  &=&\frac{1}{n}%
\sum_{i=1}^{n}h\left( i/n\right) \nu \left( g\right)  \\
&=&\lambda _{n}\left( h\right) \nu \left( g\right)  \\
&=&\left( \lambda _{n}\otimes \nu \right) \left( f\right) .
\end{eqnarray*}%
Now, let $\mathcal{F}_{\pi }\in \pi \left( UB,M\text{-}VC\right) .$ Since $%
F\equiv 1$ is an envelope of $\mathcal{F}_{\pi },$ then conditions (i) and
(ii) of Theorem \ref{gaenssler-rost-ziegler98} are easily verified. In
particular, condition (i) holds for all $p\geq 1.$ Moreover, Proposition \ref%
{stoc-bound-ran-covunum-bis} implies condition (iii) of Theorem \ref%
{gaenssler-rost-ziegler98} thus completing the proof of (\ref%
{ulln(ub)-lambdaenne}). To show that (\ref{unif-riem-integ}) implies (\ref%
{ulln(ub)}), note that, for all $p\geq 1,$%
\begin{eqnarray*}
\left\Vert \sup_{f\in \mathcal{F}_{\pi }}\left\vert \mathbb{P}_{n}\left(
f\right) -\left( \lambda \otimes \nu \right) \left( f\right) \right\vert
\right\Vert _{p} &\leq &\left\Vert \sup_{f\in \mathcal{F}_{\pi }}\left\vert
\mathbb{P}_{n}\left( f\right) -\left( \lambda _{n}\otimes \nu \right) \left(
f\right) \right\vert \right\Vert _{p} \\
&&\qquad +\left\Vert \sup_{f\in \mathcal{F}_{\pi }}\left\vert \left( \lambda
_{n}\otimes \nu \right) \left( f\right) -\left( \lambda \otimes \nu \right)
\left( f\right) \right\vert \right\Vert _{p} \\
&\triangleq &A\left( n\right) +B\left( n\right) ,
\end{eqnarray*}%
where $\left\Vert \cdot \right\Vert _{p}$ denotes the $L_{p}\left( \nu
\right) $ norm. The term $A\left( n\right) $ converges to $0$ by (\ref%
{ulln(ub)-lambdaenne}); as for the term $B\left( n\right) ,$ noting that $%
G\equiv 1$ is an envelope of $\mathcal{G},$%
\begin{eqnarray*}
B\left( n\right)  &\leq &\left\Vert \sup_{g\in \mathcal{G}}\left\vert \nu
\left( g\right) \right\vert \right\Vert _{p}\left\Vert \sup_{h\in \mathcal{H}%
}\left\vert \lambda _{n}\left( h\right) -\lambda \left( h\right) \right\vert
\right\Vert _{p} \\
&\leq &\left\Vert \sup_{h\in \mathcal{H}}\left\vert \lambda _{n}\left(
h\right) -\lambda \left( h\right) \right\vert \right\Vert _{p},
\end{eqnarray*}%
Since we have set $\nu \left( G\right) =1,$ by convention$.$ Covergence to $0
$ of $B\left( n\right) $ follows from (\ref{unif-riem-integ}). The case $%
\mathcal{F}_{\pi }\in \pi \left( \nu G,J\text{-}VC\right) $ can be treated
analogously. Conditions (i) and (ii) are verified immediately and condition
(iii) follows from the fact that $\mathcal{F}_{\pi }$ is a V.C.G.C.
\end{proof}

\bigskip

\noindent If $\mathcal{H}$ verifies (\ref{unif-riem-integ}), we will say
that $\mathcal{H}$ is \textit{uniformly Riemann-integrable.} Here is an
example of a uniformly Riemann-integrable class.

\begin{example}
\label{holderone}Let $\mathcal{H}\left( T,C,\beta \right) ,$ $T,C,\beta >0,$
the class of all H\"{o}lder functions of parameters $C$ and $\beta $ such
that $\left\vert h\left( 0\right) \right\vert \leq T.$ In other words, for
all $x_{1},x_{2}\in \left[ 0,1\right] ,$%
\begin{equation*}
\left\vert h\left( x_{1}\right) -h\left( x_{2}\right) \right\vert \leq
C\left\vert x_{1}-x_{2}\right\vert ^{\beta }.
\end{equation*}%
First of all, note that $\mathcal{H}\left( T,C,\beta \right) $ is uniformly
bounded. In fact, for all $h\in \mathcal{H}\left( T,C,\beta \right) $ and
all $x\in \left[ 0,1\right] ,$%
\begin{equation*}
\left\vert h\left( x\right) \right\vert \leq C\left\vert x\right\vert
^{\beta }+T\leq C+T.
\end{equation*}%
To show uniform Riemann-integrability, define the sequence of functions%
\begin{equation*}
h_{n}\left( x\right) \triangleq \sum_{i=1}^{n}h\left( i/n\right) \mathbf{1}%
_{\left( \frac{i-1}{n},\frac{i}{n}\right] }\left( x\right) ,\qquad n\geq 1.
\end{equation*}%
Clearly, $\lambda _{n}\left( h\right) =\lambda \left( h_{n}\right) ,$ so that%
\begin{eqnarray*}
\left\vert \lambda _{n}\left( h\right) )-\lambda \left( h\right) \right\vert
&\leq &\int_{0}^{1}\left\vert h\left( x\right) -h_{n}\left( x\right)
\right\vert \text{d}x \\
&\leq &\sum_{i=1}^{n}\int_{\frac{i-1}{n}}^{\frac{i}{n}}\left\vert h\left(
x\right) -h\left( i/n\right) \right\vert \text{d}x \\
&\leq &\frac{C}{n^{\beta }}.
\end{eqnarray*}%
Uniform Riemann-integrability follows since $\delta _{n}=C/n^{\beta }$ dos
not depend on $h.$
\end{example}

\begin{remark}
Note that continuity is not necessary to achieve uniform
Riemann-integrability. It suffices, in fact, to observe that the class%
\begin{equation*}
\mathcal{H}_{\left[ 0,1\right] }\triangleq \left\{ \mathbf{1}_{\left[ 0,t%
\right] }:0<t\leq 1\right\} ,
\end{equation*}%
is uniformly Riemann-integrable.
\end{remark}

\bigskip

The remaining part of this section is devoted to the presentation of a
U.L.L.N. in which convergence occurs almost-surely. Let $\mathcal{B}\subset
\mathcal{B}\left( \left[ 0,1\right] \right) $ be the class of all regular
Borel-sets. Recall that a Borel set $B$ is called \textit{regular }if (i) $%
\lambda \left( B\right) >0$ and (ii) $\lambda \left( \partial B\right) =0,$
where $\partial B$ is the the boudary of $B.$ By Lemma 2 in Dedecker $\left[
2001\right] ,$ if $B$ is a regular Borel set of $\left[ 0,1\right] ,$ then%
\begin{equation}
\lim_{n\rightarrow +\infty }\lambda _{n}\left( B\right) =\lambda \left(
B\right) ,  \label{dedecker}
\end{equation}%
where $\lambda _{n}$ and $\lambda $ denote, respectively, the discrete
uniform measure and the trace of Lebesgue measure on $\left[ 0,1\right] .$
Let $\left( X_{n}\right) _{n\geq 1}$ be an i.i.d. sequence of random
variables of law $\nu $ defined on a probability space $\left( \Omega ,%
\mathcal{A},P\right) $ and taking values in some measurable space $\left( U,%
\mathcal{U}\right) .$ For all $B\in \mathcal{B},$ define the $B$-\textit{%
empirical measure }on $\left( U,\mathcal{U}\right) $ by%
\begin{equation}
\nu _{n,B}\triangleq \frac{1}{card\left( B\cap \left\{ \frac{1}{n}%
,...,1\right\} \right) }\sum_{i\in B\cap \left\{ \frac{1}{n},...,1\right\}
}\delta _{X_{i}},  \label{B-indexed-em}
\end{equation}%
with the convention that $\nu _{n,B}\equiv 0$ if $B\cap \left\{ \frac{1}{n}%
,...,1\right\} =\varnothing .$

\noindent With $\mathbb{P}_{n}\triangleq n^{-1}\sum_{i=1}^{n}\delta _{\left(
i/n,X_{i}\right) }$ the sequential empirical measure, note that for $B\in
\mathcal{B}$ and $A\in \mathcal{U},$%
\begin{eqnarray*}
\mathbb{P}_{n}\left( B\times A\right)  &=&\frac{1}{n}\sum_{i\in B\cap
\left\{ \frac{1}{n},...,1\right\} }\delta _{X_{i}}\left( A\right)  \\
&=&\lambda _{n}\left( B\right) \nu _{n,B}\left( A\right) ,
\end{eqnarray*}%
so that%
\begin{equation*}
E\left( \mathbb{P}_{n}\left( B\times A\right) \right) =\lambda _{n}\left(
B\right) \nu \left( A\right) .
\end{equation*}%
It turns out that the $B$-empirical measure verifies a
Glivenko-Cantelli-type result (see Lemma \ref{glivencuccio}). This fact will
be used later to prove the following version of the U.L.L.N. for $\mathcal{F}%
_{\mathcal{B}_{\#},\mathcal{W}}$-indexed sequential empirical measure
processes where $\mathcal{F}_{\mathcal{B}_{\#},\mathcal{W}}\triangleq
\left\{ \mathbf{1}_{B}\mathbf{1}_{W}:B\in \mathcal{B}_{\#},\text{ }W\in
\mathcal{W}\right\} ,$ and $\mathcal{B}_{\#}\subset \mathcal{B}$ is a C.V.C.
of regular Borel subsets of $\left[ 0,1\right] .$

\begin{theorem}
\label{glivenko-cantelli}Let the notations of this section prevail. If $%
\mathcal{W\subset U}$ is a V.C.C. and if $\mathcal{B}_{\#}$ is a V.C.C. of
regular Borel sets of $\left[ 0,1\right] ,$ then%
\begin{equation}
P\left( \lim_{n\rightarrow +\infty }\sup_{B\in \mathcal{B}_{\#}}\sup_{W\in
\mathcal{W}}\left\vert \mathbb{P}_{n}\left( B\times W\right) -\lambda
_{n}\left( B\right) \nu \left( W\right) \right\vert =0\right) =1.
\label{gl-cant-lambaenne}
\end{equation}%
If, moreover, $\mathcal{B}_{\#}$ is such that%
\begin{equation}
\lim_{n\rightarrow +\infty }\sup_{B\in \mathcal{B}_{\#}}\left\vert \lambda
_{n}\left( B\right) -\lambda \left( B\right) \right\vert =0,
\label{dimenticata}
\end{equation}%
then%
\begin{equation}
P\left( \lim_{n\rightarrow +\infty }\sup_{B\in \mathcal{B}_{\#}}\sup_{W\in
\mathcal{W}}\left\vert \mathbb{P}_{n}\left( B\times W\right) -\lambda \left(
B\right) \nu \left( W\right) \right\vert =0\right) =1.  \label{gl-cant}
\end{equation}
\end{theorem}

\bigskip

\noindent We end this section with some comments and an open problem. First,
in the following two examples, we present special cases of V.C.C.'s of
regular Borel sets satisfying (\ref{dimenticata}). In the third example we
present a class of regular Borel-sets that is not a V.C.C. and for which (%
\ref{dimenticata}) does not hold.

\begin{example}
\label{biduejpiuuno}Fix $j\in \mathbb{N}$ and a vector $\mathbf{t}%
_{2j+1}=\left( t_{0},...,t_{2j}\right) \in \left[ 0,1\right] ^{2j+1},$ with $%
t_{2k}\leq t_{2k+1},$ for all $k=0,...,j-1$ and $t_{2k-1}<t_{2k},$ for all $%
k=1,...,j.$ We will adopt the following convention: for $j=0,$ $\mathbf{t}%
_{2j+1}=\mathbf{t}_{1}=t_{0}\triangleq t\in \left( 0,1\right] .$ Let $%
\mathfrak{T}_{2j+1}$ be the set of all such vectors. For all $\mathbf{t}%
_{2j+1}\in \mathfrak{T}_{2j+1},$ let $B_{\mathbf{t}_{2j+1}}\in \mathcal{B}%
\left( \left[ 0,1\right] \right) $ be defined by%
\begin{equation*}
B_{\mathbf{t}_{2j+1}}\triangleq \left\langle 0,t_{0}\right\rangle \cup
\left\langle t_{1},t_{2}\right\rangle \cup \cdot \cdot \cdot \cup \text{ }%
\left\langle t_{2j-1,}t_{2j}\right\rangle ,
\end{equation*}%
where $\left\langle \alpha ,\beta \right\rangle $ denotes an open, closed or
semi-open interval with extremes $\alpha <\beta $. Define the class%
\begin{equation*}
\mathcal{B}\left( 2j+1\right) \triangleq \left\{ B_{\mathbf{t}_{2j+1}}:%
\mathbf{t}_{2j+1}\in \mathfrak{T}_{2j+1}\right\} .
\end{equation*}%
In particular, $\mathcal{B}\left( 1\right) =\left\{ \left\langle
0,t\right\rangle :0<t\leq 1\right\} .$ It can be shown that $\mathcal{B}%
\left( 2j+1\right) $ is a V.C.C. of dimension $2j+1$ (see Dudley $\left[ 1999%
\right] ,$ Problem $4.11$). On the other hand, for all $B_{\mathbf{t}%
_{2j+1}}\in \mathcal{B}\left( 2j+1\right) ,$ there exists a vector $\mathbf{t%
}_{2j+1}\in \mathfrak{T}_{2j+1}$ such that%
\begin{equation*}
\lambda _{n}\left( B_{\mathbf{t}_{2j+1}}\right) =\frac{\left\lfloor
nt_{0}\right\rfloor }{n}+\sum_{k=1}^{j}\left( \frac{\left\lfloor
nt_{2k}\right\rfloor }{n}-\frac{\left\lfloor nt_{2k-1}\right\rfloor }{n}%
\right) ,
\end{equation*}%
and%
\begin{equation*}
\lambda \left( B_{\mathbf{t}_{2j+1}}\right) =t_{0}+\sum_{k=1}^{j}\left(
t_{2k}-t_{2k-1}\right) .
\end{equation*}%
It is easily seen that%
\begin{equation*}
\left\vert \lambda _{n}\left( B_{\mathbf{t}_{2j+1}}\right) -\lambda \left(
B_{\mathbf{t}_{2j+1}}\right) \right\vert \leq \frac{2\left( 2j+1\right) }{n}.
\end{equation*}%
Since $\delta _{n}=\frac{2\left( 2j+1\right) }{n}$ converges to $0$ and does
not depend on the vector $\mathbf{t}_{2j+1},$ we have that, for all $j\geq
1, $%
\begin{equation*}
\lim_{n}\sup_{B_{\mathbf{t}_{2j+1}}\in \mathcal{B}\left( 2j+1\right)
}\left\vert \lambda _{n}\left( B_{\mathbf{t}_{2j+1}}\right) -\lambda \left(
B_{\mathbf{t}_{2j+1}}\right) \right\vert =0.
\end{equation*}
\end{example}

\begin{example}
\label{biduej}Fix $j\in \mathbb{N}^{+}$ and a vector $\mathbf{t}_{2j}=\left(
t_{1},...,t_{2j}\right) \in \left[ 0,1\right] ^{2j}$ with $t_{2k}\leq
t_{2k+1}$ for all $k=1,...,j-1$ and $t_{2k-1}<t_{2k}$ for all $k=1,...,j.$
Let $\mathfrak{T}_{2j}$ be the set of all such vectors. For all $\mathbf{t}%
_{2j}\in \mathfrak{T}_{2j},$ let $B_{\mathbf{t}_{2j}}\in \mathcal{B}\left( %
\left[ 0,1\right] \right) $ be defined by%
\begin{equation*}
B_{\mathbf{t}_{2j}}=\left\langle t_{1},t_{2}\right\rangle \cup \cdot \cdot
\cdot \cup \text{ }\left\langle t_{2j-1},t_{2j}\right\rangle ,
\end{equation*}%
where the notation $\left\langle \alpha ,\beta \right\rangle $ is defined in
the previous Example. Define the class%
\begin{equation*}
\mathcal{B}\left( 2j\right) \triangleq \left\{ B_{\mathbf{t}_{2j}}:\mathbf{t}%
_{2j}\in \mathfrak{T}_{2j}\right\} .
\end{equation*}%
Again, it can be shown that $\mathcal{B}\left( 2j\right) $ is a V.C.C. of
dimension $2j$ and%
\begin{equation*}
\sup_{B_{\mathbf{t}_{2j}}\in \mathcal{B}\left( 2j\right) }\left\vert \lambda
_{n}\left( B_{\mathbf{t}_{2j}}\right) -\lambda \left( B_{\mathbf{t}%
_{2j}}\right) \right\vert \leq \frac{4j}{n}\rightarrow 0\qquad \left(
n\rightarrow +\infty \right) .
\end{equation*}
\end{example}

\bigskip

\noindent Here is a class of regular Borel sets, closely related to those
treated in Examples \ref{biduejpiuuno} and \ref{biduej}, for which (\ref%
{dimenticata}) does not hold

\begin{example}
\label{biinfinito}Define the class%
\begin{equation*}
\mathcal{B}_{\infty }\triangleq \bigcup\limits_{n\geq 1}\mathcal{B}\left(
n\right) ,
\end{equation*}%
where $\mathcal{B}\left( n\right) $ is defined in Example \ref{biduejpiuuno}
or \ref{biduej} according to whether $n$ is odd or even, respectively. To
see that (\ref{dimenticata}) does not hold for $\mathcal{B}_{\infty },$
consider the sequence $\left( B_{n}\right) _{n\geq 1}$ of Borel sets defined
by%
\begin{equation*}
B_{n}=\left( 0,\frac{1}{n}-\varepsilon _{n}\right] \cup \left( \frac{1}{n},%
\frac{2}{n}-\varepsilon _{n}\right] \cup \cdot \cdot \cdot \cup \text{ }%
\left( \frac{n-1}{n},1-\varepsilon _{n}\right] ,\qquad n\geq 1,
\end{equation*}%
where $\varepsilon _{n}=\frac{1}{n2^{n}},$ $n\geq 1.$ Clearly, for all $%
n\geq 1,$ $B_{n}\in \mathcal{B}\left( n\right) ,$ $\lambda _{n}\left(
B_{n}\right) =0$ and $\lambda \left( B_{n}\right) =1-n\varepsilon _{n}.$
Consequently,%
\begin{equation*}
\lim_{n\rightarrow +\infty }\left\vert \lambda _{n}\left( B_{n}\right)
-\lambda \left( B_{n}\right) \right\vert =\lim_{n\rightarrow +\infty }\left(
1-n\varepsilon _{n}\right) =1.
\end{equation*}%
Finally, note that $\mathcal{B}_{\infty }$ is clearly not a V.C.C.
\end{example}

\bigskip

\noindent Here is the statement of the announced open problem:

\bigskip

\textbf{Problem. }Is it possible to establish a U.L.L.N. along the lines of
Theorem \ref{glivenko-cantelli} for sequential empirical measure processes
indexed by the class%
\begin{equation*}
\mathcal{F}_{\mathcal{B}^{\ast },\mathcal{W}}\triangleq \left\{ f=\mathbf{1}%
_{B}\mathbf{1}_{W}:B\in \mathcal{B}^{\ast },\text{ }W\in \mathcal{W}\right\}
,
\end{equation*}%
where $\mathcal{W\subset U}$ is a V.C.C. and where $\mathcal{B}^{\ast }$ is
an arbitrary class of regular\textit{\ }Borel sets (not necessarily a
C.V.C.)?

\bigskip

\noindent Clearly, if $\mathcal{B}^{\ast }$ is not a V.C.C., Theorem \ref%
{glivenko-cantelli} does not apply. Moreover, our techniques are not
conclusive. In fact, rehearsing the proof in section \ref{theorem5}, one is
confronted with studying the convergence of the double series%
\begin{equation*}
S\left( D,c\right) =\sum_{n\geq 1}\sum_{k=1}^{n}k^{D}\binom{n}{k}\exp \left(
-cn\right) ,
\end{equation*}%
where $D$ is a strictly positive natural number and $c$ is a positive real
number. Now, if $c>\log 2,$%
\begin{equation*}
S\left( D,c\right) \leq \sum_{n\geq 1}n^{D}\left( \frac{2}{e^{c}}\right)
^{n}<+\infty ,
\end{equation*}%
whereas, if $c\leq \log 2,$%
\begin{equation*}
S\left( D,c\right) \geq \sum_{n\geq 1}\left( \frac{2}{e^{c}}\right)
^{n}=+\infty .
\end{equation*}

\section{Functional central limit theorems\label{fclt's}}

This section is devoted to obtaining functional central limit theorems for
sequences of $\mathcal{F}_{\pi }$-indexed sequential empirical measure
processes, when $\mathcal{F}_{\pi }\in \pi \left( \nu G^{2},J\text{-}%
VC\right) $ and $\mathcal{F}_{\pi }\in \pi \left( UB,M\text{-}VC\right) .$
Obtaining F.C.L.T's amounts essentially to proving two facts: (i)
convergence of finite dimentional distributions of the sequence to those of
a centered Gaussian process and (ii) asymptotic equicontinuity of the
sequence of processes. We start by showing convergence of finite-dimensional
laws. Note that this will be done for a much larger class than those mainly
considered in this paper. The reason for this is to exhibit the limiting
Gaussian process one needs to consider in view of possible generalizations
of our results. Let $\mathcal{Q}$ be the class of all $\left( \mathcal{B}%
\left( \left[ 0,1\right] \right) \otimes \mathcal{U}\right) $-measurable
functions $q\in \mathbb{R}^{\left[ 0,1\right] \times U}$ such that the
following two conditions hold:

\begin{enumerate}
\item[(a)] for every $q\in \mathcal{Q},$ there exists a function $g_{q}\in
L_{2}\left( \nu \right) $ such that, for all $\left( s,x\right) \in \left[
0,1\right] \times U,$%
\begin{equation*}
\left\vert q\left( s,x\right) \right\vert \leq g_{q}\left( x\right) ;
\end{equation*}

\item[(b)] for all $q\in \mathcal{Q}$ and all $x\in U,$ the function%
\begin{equation*}
s\mapsto q\left( s,x\right) :\left[ 0,1\right] \rightarrow \mathbb{R},
\end{equation*}%
is $\lambda $-almost everywhere continuous.
\end{enumerate}

\noindent Consider the sequence of $\mathcal{Q}$-indexed s.e.m.p.'s $\left\{
Z_{n}\left( q\right) :q\in \mathcal{Q}\right\} ,$ $n\geq 1,$ where, for each
$q\in \mathcal{Q}$ and $n\geq 1,$ $Z_{n}\left( q\right) $ is defined by%
\begin{equation*}
Z_{n}\left( q\right) \triangleq \sqrt{n}\left( \mathbb{P}_{n}\left( q\right)
-\left( \lambda _{n}\otimes \nu \right) \left( q\right) \right) ,
\end{equation*}%
the quantity $\mathbb{P}_{n}\left( q\right) $ being defined in (\ref%
{iniziale}) and $\lambda _{n}$ denoting the discrete uniform measure on $%
\left[ 0,1\right] $. Let $\left\{ Z\left( q\right) :q\in \mathcal{Q}\right\}
$ be a centered Gaussian process whose covariance structure is given by%
\begin{equation}
Cov\left( Z\left( q_{1}\right) ,Z\left( q_{2}\right) \right) =\int_{\left[
0,1\right] }\left[ \nu \left( q_{1}\cdot q_{2}\right) \left( s\right) -\nu
\left( q_{1}\right) \left( s\right) \cdot \nu \left( q_{2}\right) \left(
s\right) \right] \lambda \left( \text{d}s\right) ,  \label{kieferissimo}
\end{equation}%
where $\lambda $ is the trace of Lebesgue measure on $\left[ 0,1\right] ,$
where%
\begin{equation*}
\nu \left( q_{1}\cdot q_{2}\right) \left( s\right) \triangleq
\int_{U}q_{1}\left( s,x\right) q_{2}\left( s,x\right) \nu \left( \text{d}%
x\right) ,
\end{equation*}%
and where%
\begin{equation*}
\nu \left( q_{p}\right) \left( s\right) \triangleq \int_{U}q_{p}\left(
s,x\right) \nu \left( \text{d}x\right) ,\qquad p=1,2.
\end{equation*}%
Note that if $q_{1}=\mathbf{1}_{\left[ 0,s\right] },$ $0\leq s\leq 1$ and $%
q_{2}=\mathbf{1}_{\left[ 0,x\right] }$ $0\leq x\leq 1,$ then the covariance
structure described in (\ref{kieferissimo}) defines that of the classical
Kiefer process on $\left[ 0,1\right] ^{2}.$ In the next section, we will
prove the following

\begin{proposition}
\label{finitodimensionali}The finite dimensional laws of the sequence of $%
\mathcal{Q}$-indexed s.e.m.p.'s $\left\{ Z_{n}\left( q\right) :q\in \mathcal{%
Q}\right\} ,$ $n\geq 1,$ converge to those of the centered Gaussian process $%
\left\{ Z\left( q\right) :q\in \mathcal{Q}\right\} $ whose covariance
structure is given by (\ref{kieferissimo}).
\end{proposition}

\noindent We will first use the Lindeberg central limit theorem to prove
that, for all $q\in \mathcal{Q},$ the sequence of random variables $%
Z_{n}\left( q\right) ,$ $n\geq 1,$ converges in law to $Z\left( q\right) .$
The proof of Proposition \ref{finitodimensionali} will be then completed by
an application of the Cram\'{e}r-Wold device. In what follows, we analyze
more closely the defining conditions of the class $\mathcal{Q}.$ First of
all, we observe that for all $q\in \mathcal{Q},$ the function%
\begin{equation*}
\nu \left( q^{2}\right) \left( s\right) \triangleq \int_{U}q^{2}\left(
s,x\right) \nu \left( \text{d}x\right) ,
\end{equation*}%
is Riemann-integrable. In fact, condition (b) implies that, for all $1\leq
p<+\infty ,$ the function $s\mapsto q^{p}\left( s,x\right) $ is $\lambda $%
-almost everywhere continuous for all $x\in U.$ This, in turn, implies the $%
\lambda $-almost everywhere continuity of%
\begin{equation*}
\nu \left( q^{p}\right) \left( s\right) \triangleq \int_{U}q^{p}\left(
s,x\right) \nu \left( dx\right) ,
\end{equation*}%
for all $1\leq p<+\infty $ and in particular for $p=2.$ Now, condition (a)
implies the boundedness of $\nu \left( q^{2}\right) :\left[ 0,1\right]
\rightarrow \mathbb{R},$ since this conditions ensures the existence of a
function $g_{q}:U\rightarrow \mathbb{R}$ such that, for all $s\in \left[ 0,1%
\right] ,$%
\begin{equation*}
\int_{U}q^{2}\left( s,x\right) \nu \left( \text{d}x\right) \leq
\int_{U}g_{q}^{2}\left( x\right) \nu \left( \text{d}x\right) =\nu \left(
g_{q}^{2}\right) <+\infty .
\end{equation*}%
Riemann-integrability of $\nu \left( q^{2}\right) \left( \cdot \right) $
follows by Theorem 1.7.1 in Ash $\left[ 1972\right] .$ This fact has an
important consequence for our purposes. Namely%
\begin{equation}
\lim_{n\rightarrow +\infty }\left( \lambda _{n}\otimes \nu \right) \left(
q^{2}\right) =\int_{\left[ 0,1\right] \times U}q^{2}\left( s,x\right) \left(
\lambda \otimes \nu \right) \left( \text{d}s,\text{d}x\right) \triangleq
\left( \lambda \otimes \nu \right) \left( q^{2}\right) ,  \label{namely}
\end{equation}%
a condition without which applying Lindeberg central limit theorem would be
impossible. Moreover, $\mathcal{Q}$ is a linear space. In other words, for
all $\left( a_{1},...,a_{K}\right) \in \mathbb{R}^{K}$ and all $\left(
g_{1},...,g_{K}\right) \in \mathcal{Q}^{K},$ $K=1,2,...,$ the function $%
q_{\Sigma }\triangleq \sum_{j=1}^{K}a_{j}q_{j}\left( s,x\right) $ is an
element of $\mathcal{Q}$. In fact, it is immediately seen that $q_{\Sigma }$
verifies condition (b). As for condition (a), observe that, for all $\left(
s,x\right) \in \left[ 0,1\right] \times U,$%
\begin{equation*}
\left\vert \sum_{j=1}^{K}a_{j}q_{j}\left( s,x\right) \right\vert \leq
M_{\Sigma }\sum_{j=1}^{K}\left\vert g_{q_{j}}\left( x\right) \right\vert ,
\end{equation*}%
where $M_{\Sigma }\triangleq \max_{1\leq j\leq K}\left\vert a_{j}\right\vert
$ is a positive (finite) constant and where $g_{q_{1}},...,g_{q_{K}}$ are
elements $L_{2}\left( \nu \right) $ whose existence is guaranteed by
condition (b). Cauchy-Scwartz inequality now gives%
\begin{equation*}
\nu \left( g_{\Sigma }^{2}\right) =\sum_{j=1}^{K}\nu \left( g_{j}^{2}\right)
+\sum_{1\leq l\neq m\leq K}\nu \left( \left\vert g_{q_{l}}\right\vert
\left\vert g_{q_{m}}\right\vert \right) <+\infty ,
\end{equation*}%
proving that $q_{\Sigma }$ verifies condition (a). Once more, this is key to
employing the Cram\'{e}r-Wold device.

\bigskip

Let us return to $\mathcal{F}_{\pi }$-indexed s.e.m.p.'s with $\mathcal{F}%
_{\pi }\in \pi \left( \nu G^{2},J\text{-}VC\right) $ or $\mathcal{F}_{\pi
}\in \pi \left( UB,M\text{-}VC\right) .$ It is trivial to see that, in both
cases, $\mathcal{F}_{\pi }\subset \mathcal{Q}.$ Now for all $f_{1},f_{2}\in
\mathcal{F}_{\pi }=\mathcal{H\cdot G},$ write $f_{1}=h_{1}g_{1}$ and $%
f_{2}=h_{2}g_{2},$ with $h_{1},h_{2}\in \mathcal{H}\subset \mathfrak{R}^{%
\left[ 0,1\right] }$ and $g_{1},g_{2}\in \mathcal{G}\subset \mathfrak{R}^{U},
$ to see that the covariance structure of the $\mathcal{F}_{\pi }$-indexed
limiting centered Gaussian process $\left\{ Z\left( f\right) :f\in \mathcal{F%
}_{\pi }\right\} $ is given by%
\begin{equation}
Cov\left( Z\left( f_{1}\right) ,Z\left( f_{2}\right) \right) =\lambda \left(
h_{1}h_{2}\right) \left[ \nu \left( g_{1}g_{2}\right) -\nu \left(
g_{1}\right) \nu \left( g_{2}\right) \right] .  \label{prodotti}
\end{equation}%
In the next section we will prove weak convergence of the sequence of
s.e.m.p.'s $\left\{ Z_{n}\left( f\right) :f\in \mathcal{F}_{\pi }\right\} ,$
$n\geq 1,$ to a centered Gaussian process with uniformly bounded and
uniformly continuous sample paths and covariance function given by (\ref%
{prodotti}), both when $\mathcal{F}_{\pi }\in \pi \left( \nu G^{2},J\text{-}%
VC\right) $ and $\mathcal{F}_{\pi }\in \pi \left( UB,M\text{-}VC\right) .$
By "weak convergence" we mean in the sense of Hoffman-J\o rgensen $\mathcal{L%
}$-convergence (see Hoffmann J\o rgensen $\left[ 1991\right] $). Observe, in
fact, that the sequence of processes $Z_{n},$ $n\geq 1,$ can be seen as a
sequence of random quantities (i.e. not necessarily measurable) with sample
paths in the pseudo-metric space $\left( \ell ^{\infty }\left( \mathcal{F}%
_{\pi }\right) ,\left\Vert \cdot \right\Vert _{\mathcal{F}_{\pi }}\right) ,$
where $\left\Vert \cdot \right\Vert _{\mathcal{F}_{\pi }}$ denotes the sup
norm defined by $\left\Vert \cdot \right\Vert _{\mathcal{F}_{\pi
}}\triangleq \sup_{f\in \mathcal{F}_{\pi }}\left\vert \cdot \right\vert .$
Towards our aim, we need to show that there exists a version of the limiting
Gaussian process with uniformly bounded and uniformly continuous sample
paths, a fact for which one only needs that the sub-space $\left(
UB^{b}\left( \mathcal{F}_{\pi },d\right) ,\left\Vert \cdot \right\Vert _{%
\mathcal{F}_{\pi }}\right) $ of uniformly bounded and uniformly continuous
functions (with respect to some metric $d$) to be a separable subset of $%
\left( \ell ^{\infty }\left( \mathcal{F}_{\pi }\right) ,\left\Vert \cdot
\right\Vert _{\mathcal{F}_{\pi }}\right) .$ For this, it is suffices that $%
\left( \mathcal{F}_{\pi },d\right) $ be a totally bounded pseudometric
space. That is why, from now on, $\mathcal{F}_{\pi }$ will be endowed with
the totally bounded pseudometric (see the section \ref{ancillary} for the
proof)%
\begin{equation*}
d\left( f_{1},f_{2}\right) =d\left( h_{1}g_{1},h_{2}g_{2}\right) \triangleq
d_{\lambda }^{\left( 2\right) }\left( h_{1},h_{2}\right) +d_{\nu }^{\left(
2\right) }\left( g_{1},g_{2}\right) ,
\end{equation*}%
where $d_{\lambda }^{\left( 2\right) }$ and $d_{\nu }^{\left( 2\right) }$
denote, respectively, the $L_{2}\left( \lambda \right) $ pseudometric on $%
\mathcal{H}$ and the $L_{2}\left( \nu \right) $ pseudo metric on $\mathcal{G}%
.$ Once this result is at hand, we will need to prove asymptotic ($d$%
-)equicontinuity of the sequence $Z_{n},$ $n\geq 1.$ We will denote the fact
that $Z_{n}$ converges weakly to a separable centres Gaussian process $Z$
with uniformly bounded and uniformly continuous sample paths by writing, as
in Ziegler $\left[ 1997\right] ,$ $Z_{n}\underset{sep}{\overset{\mathcal{L}}{%
\rightarrow }}Z$ in $\ell ^{\infty }\left( \mathcal{F}_{\pi },d\right) .$ In
the next section we will prove the following results:

\begin{theorem}
\label{fclt_nuggi}If $\mathcal{H\cdot G}\triangleq \mathcal{F}_{\pi }\in \pi
\left( \nu G^{2},J\text{-}VC\right) $ and if $\mathcal{H}$ is such that%
\begin{equation}
\lim_{n\rightarrow +\infty }\sup_{h_{1},h_{2}\in \mathcal{H}}\left\vert
\lambda _{n}\left( \left( h_{1}-h_{2}\right) ^{2}\right) -\lambda \left(
\left( h_{1}-h_{2}\right) ^{2}\right) \right\vert =0,  \label{oscillazioni}
\end{equation}%
then%
\begin{equation*}
Z_{n}\underset{sep}{\overset{\mathcal{L}}{\rightarrow }}Z\text{ in }\ell
^{\infty }\left( \mathcal{F}_{\pi },d\right) .
\end{equation*}
\end{theorem}

\begin{theorem}
\label{fclt_ubbi}If $\mathcal{H\cdot G}\triangleq \mathcal{F}_{\pi }\in \pi
\left( UB,M\text{-}VC\right) $ and if $\mathcal{H}$ verifies (\ref%
{oscillazioni}), then%
\begin{equation*}
Z_{n}\underset{sep}{\overset{\mathcal{L}}{\rightarrow }}Z\text{ in }\ell
^{\infty }\left( \mathcal{F}_{\pi },d\right) .
\end{equation*}
\end{theorem}

\noindent As will be seen in the next section, Theorem \ref{fclt_nuggi} is
in fact a corollary of the findings in Section 4.2 of Ziegler $\left[ 1997%
\right] .$ As for the proof of Theorem \ref{fclt_ubbi}, we will use a
maximal inequality stated in Theorem 3.1 of the same reference. We end this
section presenting two examples of classes $\mathcal{H}\subset \mathfrak{R}^{%
\left[ 0,1\right] }$ for which (\ref{oscillazioni}) holds.

\begin{example}
\label{suff-oscillazioni-uno}It is easily seen that if $\mathcal{H}$ is (1)
uniformly Riemann-integrable, (2) such that%
\begin{equation*}
\mathcal{H}^{2}\triangleq \left\{ h^{2}:h\in \mathcal{H}\right\}
\end{equation*}%
is uniformly Riemann-interable and (3) such that for all $h_{1},h_{2}\in
\mathcal{H},$ $h\triangleq h_{1}h_{2}$ is an element of $\mathcal{H},$ then
condition (\ref{oscillazioni}) holds.
\end{example}

\begin{example}
\label{hokderissimo}We show that the class $\mathcal{H}\left( T,C,\beta
\right) $ defined in Example \ref{holderone} satisfies (\ref{oscillazioni}).
Incidentally, observe that condition (3) of Example \ref%
{suff-oscillazioni-uno} does not hold for $\mathcal{H}\left( T,C,\beta
\right) .$ Now, for all $h_{1},h_{2}\in \mathcal{H}\left( T,C,\beta \right) ,
$ define the two respective function sequences%
\begin{eqnarray*}
h_{1,n}\left( x\right)  &\triangleq &\sum_{i=1}^{n}h_{1}\left( i/n\right)
\mathbf{1}_{\left( \frac{i-1}{n},\frac{i}{n}\right] }\left( x\right) ,\qquad
n\geq 1, \\
h_{1,n}\left( x\right)  &\triangleq &\sum_{i=1}^{n}h_{2}\left( i/n\right)
\mathbf{1}_{\left( \frac{i-1}{n},\frac{i}{n}\right] }\left( x\right) ,\qquad
n\geq 1.
\end{eqnarray*}%
Then, for all $n\geq 1,$%
\begin{equation*}
\left( h_{1,n}\left( x\right) -h_{1,n}\left( x\right) \right) =\sum_{i=1}^{n}
\left[ h_{1}\left( i/n\right) -h_{2}\left( i/n\right) \right] \mathbf{1}%
_{\left( \frac{i-1}{n},\frac{i}{n}\right] }\left( x\right) ,
\end{equation*}%
and%
\begin{equation*}
\left( h_{1,n}\left( x\right) -h_{1,n}\left( x\right) \right)
^{2}=\sum_{i=1}^{n}\left[ h_{1}\left( i/n\right) -h_{2}\left( i/n\right) %
\right] ^{2}\mathbf{1}_{\left( \frac{i-1}{n},\frac{i}{n}\right] }\left(
x\right) .
\end{equation*}%
It follows that%
\begin{equation*}
\lambda _{n}\left( \left( h_{1}-h_{2}\right) ^{2}\right) =\frac{1}{n}%
\sum_{i=1}^{n}\left[ h_{1}\left( i/n\right) -h_{2}\left( i/n\right) \right]
^{2}=\lambda \left( \left( h_{1,n}\left( x\right) -h_{1,n}\left( x\right)
\right) ^{2}\right) .
\end{equation*}%
Note that%
\begin{eqnarray*}
\sup_{h_{1},h_{2}\in \mathcal{H}\left( T,C,\beta \right) }\left\vert \lambda
\left( \left( h_{1}-h_{2}\right) ^{2}-\left( h_{1,n}\left( x\right)
-h_{1,n}\left( x\right) \right) ^{2}\right) \right\vert  &\leq &2\sup_{h\in
\mathcal{H}\left( T,C,\beta \right) }\left\vert \lambda \left(
h_{n}^{2}-h^{2}\right) \right\vert  \\
&&\qquad +\text{ }2\sup_{h_{1},h_{2}\in \mathcal{H}\left( T,C,\beta \right)
}\left\vert \lambda \left( h_{1,n}h_{2,n}-h_{1}h_{2}\right) \right\vert .
\end{eqnarray*}%
Since, by Cauchy-Schwartz inequality and uniform boudedness of $\mathcal{H}%
\left( T,C,\beta \right) $ (see Example \ref{holderone}),%
\begin{eqnarray*}
\left\vert \lambda \left( h_{n}^{2}-h^{2}\right) \right\vert  &\leq &\sqrt{%
\lambda \left( h_{n}^{2}\right) \lambda \left( \left( h_{n}-h\right)
^{2}\right) }+\sqrt{\lambda \left( h^{2}\right) \lambda \left( \left(
h_{n}-h\right) ^{2}\right) } \\
&\leq &2\left( C+T\right) \sqrt{\lambda \left( \left( h_{n}-h\right)
^{2}\right) },
\end{eqnarray*}%
and since, for all $h\in \mathcal{H}\left( T,C,\beta \right) $ and all $%
x_{1},x_{2}\in \left[ 0,1\right] ,$%
\begin{equation*}
\left\vert h\left( x_{1}\right) -h\left( x_{2}\right) \right\vert ^{2}\leq
C^{2}\left\vert x_{1}-x_{2}\right\vert ^{2\beta },
\end{equation*}%
proceeding exactly as in Example \ref{holderone}, one has%
\begin{equation*}
\lambda \left( \left( h_{n}-h\right) ^{2}\right) \leq \frac{C^{2}}{n^{2\beta
}},
\end{equation*}%
so that%
\begin{equation*}
\lim_{n\rightarrow +\infty }\sup_{h\in \mathcal{H}\left( T,C,\beta \right)
}\left\vert \lambda \left( h_{n}^{2}-h^{2}\right) \right\vert =0.
\end{equation*}%
Analogously, it is possible to show that%
\begin{eqnarray*}
\left\vert \lambda \left( h_{1,n}h_{2,n}-h_{1}h_{2}\right) \right\vert
&\leq &\sqrt{\lambda \left( h_{2,n}^{2}\right) \lambda \left( \left(
h_{1,n}-h_{1}\right) ^{2}\right) }+\sqrt{\lambda \left( h_{2}^{2}\right)
\lambda \left( \left( h_{2,n}-h_{2}\right) ^{2}\right) } \\
&\leq &\frac{2C\left( C+T\right) }{n^{\beta }},
\end{eqnarray*}%
thus proving%
\begin{equation*}
\lim_{n\rightarrow +\infty }\sup_{h_{1},h_{2}\in \mathcal{H}\left( T,C,\beta
\right) }\left\vert \lambda \left( h_{1,n}h_{2,n}-h_{1}h_{2}\right)
\right\vert =0.
\end{equation*}%
We have shown that (\ref{oscillazioni}) holds for $\mathcal{H}\left(
T,C,\beta \right) .$
\end{example}

\section{Proofs}

\subsection{Ancillary lemmas\label{ancillary}}

Here are the definition of covering numbers of a pseudo-metric space and the
statements of four elementary facts regarding such quantities.

\begin{definition}
\label{covering-numbers-bis}Let $\left( M,d\right) $ be a pseudo-metric
space. Given $u>0,$ a set
\begin{equation*}
M^{\left( n\right) }=\left\{ m_{1},...,m_{n}\right\} \subset M
\end{equation*}%
is called a $u$-net in $\left( M,d\right) $ if, for all $m\in M,$ there
exists $m_{i}\in \left\{ m_{1},...,m_{n}\right\} $ such that $d\left(
m,m_{i}\right) <u.$ The number%
\begin{equation*}
N\left( u,M,d\right) \triangleq \inf \left\{ n\geq 1:\text{there exists a }u%
\text{-net }M^{\left( n\right) }\text{ in }\left( M,d\right) \right\} ,
\end{equation*}%
is called $u$-covering number of $\left( M,d\right) .$
\end{definition}

\begin{lemma}
\label{cov-maj-bis}Let $\left( M,d\right) $ be a metric space and let $%
M^{\prime }\subset M.$ Then, for all $u>0,$%
\begin{equation*}
N\left( u,M^{\prime },d\right) \leq N\left( u,M,d\right) .
\end{equation*}
\end{lemma}

\begin{lemma}
\label{cov-maj-2-bis}Let $d$ and $d^{\prime }$ be two pseudo-metrics on some
set $M$ and suppose that, for all $m_{1},m_{2}\in M,$ $d\left(
m_{1},m_{2}\right) \leq d^{\prime }\left( m_{1},m_{2}\right) .$ Then, for
all $u>0,$%
\begin{equation*}
N\left( u,M,d\right) \leq N\left( u,M,d^{\prime }\right) .
\end{equation*}
\end{lemma}

\begin{lemma}
\label{cov-sum-bis}Let $\left( M_{1},d_{1}\right) $ and $\left(
M_{2},d_{2}\right) $ be two psudo-metric spaces. Define the pseudo metric
space $\left( M,d\right) $ by $M\triangleq M_{1}\times M_{2}$ and $%
d\triangleq d_{1}+d_{2}.$ Then, for all $u>0$ and all $t\in \left(
0,1\right) ,$%
\begin{equation*}
N\left( u,M,d\right) \leq N\left( tu,M_{1},d_{1}\right) N\left( \left(
1-t\right) u,M_{2},d_{2}\right) .
\end{equation*}
\end{lemma}

\begin{lemma}
\label{cov-bij-bis}Let $\left( M,d\right) $ and $\left( M^{\prime
},d^{\prime }\right) $ be two pseudo-metric spaces. Let $b:M\rightarrow M$
be a bijection such that for all $m_{1}^{\prime },m_{2}^{\prime }\in
M^{\prime },$%
\begin{equation*}
d^{\prime }\left( m_{1}^{\prime },m_{2}^{\prime }\right) =d\left(
b^{-1}\left( m_{1}^{\prime }\right) ,b^{-1}\left( m_{2}^{\prime }\right)
\right) .
\end{equation*}%
Then, for all $u>0,$%
\begin{equation*}
N\left( u,M,d\right) =N\left( u,M^{\prime },d^{\prime }\right) .
\end{equation*}
\end{lemma}

\noindent The remaining part of this section is devoted to the proof of some
entropy properties of classes $\mathcal{F}_{\pi }\in \pi \left( UB,M\text{-}%
VC\right) $ and $\mathcal{F}_{\pi }\in \pi \left( \nu G^{2},J\text{-}%
VC\right) .$ We start by proving stochastic boudedness of some random
covering numbers when $\mathcal{F}_{\pi }\in \pi \left( UB,M\text{-}%
VC\right) $. This fact is used to show that $\mathcal{F}_{\pi }$-indexed
s.e.m.p. verifies an $L_{p}$-uniform law of large numbers (see Corollary \ref%
{corollary-ulln}). Let $f,f_{1},f_{2}\in \mathcal{F}_{\pi }\in \pi \left(
UB,M\text{-}VC\right) $ and write $f=hg,$ $f_{i}=h_{i}g_{i},$ $i=1,2,$ for
some $h,h_{1},h_{2}\in \mathcal{H}$ and $g,g_{1},g_{2}\in \mathcal{G}.$
Define the quantities%
\begin{equation*}
\left\Vert f\right\Vert _{\mathbb{P}_{n}}\triangleq \mathbb{P}_{n}\left(
\left\vert f\right\vert \right) =\frac{1}{n}\sum_{i=1}^{n}\left\vert h\left(
i/n\right) g\left( X_{i}\right) \right\vert ,
\end{equation*}%
and%
\begin{eqnarray*}
d_{\mathbb{P}_{n}}^{\left( 1\right) }\left( f_{1},f_{2}\right) &\triangleq
&\left\Vert f_{1}-f_{2}\right\Vert _{\mathbb{P}_{n}} \\
&=&\frac{1}{n}\sum_{i=1}^{n}\left\vert h_{1}\left( i/n\right) g_{1}\left(
X_{i}\right) -h_{2}\left( i/n\right) g_{2}\left( X_{i}\right) \right\vert .
\end{eqnarray*}

\begin{proposition}
\label{stoc-bound-ran-covunum-bis}If $\mathcal{F}_{\pi }\in \pi \left( UB,M%
\text{-}VC\right) ,$ then, for all $\tau >0,$ there exists a constant $C,$
depending only on $\tau $ and the V.C. dimensions of classes $\mathcal{H}$
and $\mathcal{G},$ such that, for all $\omega \in \Omega $ and all $n\geq 1,$
\begin{equation*}
\sup_{\omega \in \Omega }\sup_{n\geq 1}\text{ }N\left( \tau ,\mathcal{F}%
_{\pi },d_{\mathbb{P}_{n}}^{\left( 1\right) }\right) \leq C
\end{equation*}
\end{proposition}

\begin{proof}
Let $\nu _{n}$ and $\lambda _{n}$ denote, respectively, the (classical)
empirical measure defined by $\nu _{n}\triangleq n^{-1}\sum_{i=1}^{n}\delta
_{X_{i}}$ and the discrete uniform measure on $\left[ 0,1\right] $ defined
by $\lambda _{n}\triangleq n^{-1}\sum_{i=1}^{n}\delta _{\left\{ i/n\right\}
}.$ For all $f_{1}=h_{1}g_{1}$ and $f_{2}=h_{2}g_{2}$ in $\mathcal{F}_{\pi
}\in \pi \left( UB,M\text{-}VC\right) ,$ by uniform boudedness of classes $%
\mathcal{H}$ and $\mathcal{G},$ it is immediately seen that, for all $\omega
\in \Omega $ and all $n\geq 1,$%
\begin{equation*}
d_{\mathbb{P}_{n}}^{\left( 1\right) }\left( f_{1},f_{2}\right) \leq d_{\nu
_{n}}^{\left( 1\right) }\left( g_{1},g_{2}\right) +d_{\lambda _{n}}^{\left(
1\right) }\left( h_{1},h_{2}\right) ,
\end{equation*}%
where $d_{\nu _{n}}^{\left( 1\right) }\left( g_{1},g_{2}\right) \triangleq
\nu _{n}\left( \left\vert g_{1}-g_{2}\right\vert \right) $ and $d_{\lambda
_{n}}^{\left( 1\right) }\left( h_{1},h_{2}\right) \triangleq \lambda
_{n}\left( \left\vert h_{1}-h_{2}\right\vert \right) .$ By Lemma \ref%
{cov-maj-2-bis}, we have that, for all $\tau >0,$ all $\omega \in \Omega $
and all $n\geq 1,$%
\begin{equation}
N\left( \tau ,\mathcal{F}_{\pi },d_{\mathbb{P}_{n}}^{\left( 1\right)
}\right) \leq N\left( \tau ,\mathcal{F}_{\pi },d_{\nu _{n}}^{\left( 1\right)
}+d_{\lambda _{n}}^{\left( 1\right) }\right) .  \label{prima-bis}
\end{equation}%
Equip $\mathcal{H\times G}$ with the equivalence relation $\sim $ defined by%
\begin{equation*}
\left( h,g\right) \sim \left( h^{\prime },g^{\prime }\right) \qquad
\Leftrightarrow \qquad hg\equiv h^{\prime }g^{\prime },
\end{equation*}%
and let $\widetilde{\mathcal{H\times G}}$ be the subset of $\mathcal{H\times
G}$ obtained by choosing exactly one element from each equivalence class.
Now the application $b:\widetilde{\mathcal{H\times G}}$ $\rightarrow $ $%
\mathcal{F}_{\pi }:\left( h,g\right) \mapsto hg$ is a bijection so that,
applying Lemma \ref{cov-bij-bis} then Lemma \ref{cov-maj-bis}, one obtains
that, for all $\tau >0,$ all $\omega \in \Omega $ and all $n\geq 1,$%
\begin{equation}
N\left( \tau ,\mathcal{F}_{\pi },d_{\nu _{n}}^{\left( 1\right) }+d_{\lambda
_{n}}^{\left( 1\right) }\right) =N\left( \tau ,\widetilde{\mathcal{H\times G}%
},d_{\nu _{n}}^{\left( 1\right) }+d_{\lambda _{n}}^{\left( 1\right) }\right)
\leq N\left( \tau ,\mathcal{H\times G},d_{\nu _{n}}^{\left( 1\right)
}+d_{\lambda _{n}}^{\left( 1\right) }\right)  \label{seconda-bis}
\end{equation}%
Lemma \ref{cov-sum-bis} with $t=1/2$ now gives for all $\omega \in \Omega $
and all $n\geq 1,$%
\begin{equation}
N\left( \tau ,\mathcal{H\times G},d_{\nu _{n}}^{\left( 1\right) }+d_{\lambda
_{n}}^{\left( 1\right) }\right) \leq N\left( \tau /2,\mathcal{H},d_{\lambda
_{n}}^{\left( 1\right) }\right) N\left( \tau /2,\mathcal{G},d_{\nu
_{n}}^{\left( 1\right) }\right) .  \label{terza-bis}
\end{equation}%
Since $\mathcal{H}$ is a V.C.G.C. and $\lambda _{n}$ is a finite measure for
all $n\geq 1,$ for all $\varepsilon >0,$ there exists a constant $C_{1}$
depending only on $\varepsilon $ and on the V.C. dimension of $\mathcal{H}$
such that, for all $n\geq 1,$ $N\left( \varepsilon ,\mathcal{H},d_{\lambda
_{n}}^{\left( 1\right) }\right) \leq C_{1}.$ Analogously, since $\mathcal{G}$
is a V.C.G.C. and $\nu _{n}$ is a finite measure for all $\omega \in \Omega $
and all $n\geq 1,$ for all $\varepsilon >0,$ there exists a constant $C_{2}$
depending only on $\varepsilon $ and the V.C. dimension of $\mathcal{G}$
such that, for all $\omega \in \Omega $ and $n\geq 1,$ $N\left( \varepsilon ,%
\mathcal{G},d_{\nu _{n}}^{\left( 1\right) }\right) \leq C_{2}.$ Take $\tau
=2\varepsilon .$ Combining inequalities (\ref{prima-bis}), (\ref{seconda-bis}%
) and (\ref{terza-bis}) gives%
\begin{equation*}
\sup_{\omega \in \Omega }\sup_{n\geq 1}\text{ }N\left( \tau ,\mathcal{F}%
_{\pi },d_{\mathbb{P}_{n}}^{\left( 1\right) }\right) \leq C_{1}C_{2}.
\end{equation*}%
The proof is complete with $C=C_{1}C_{2}.$
\end{proof}

\bigskip

We now turn to the case $\mathcal{F}_{\pi }\in \pi \left( \nu G^{2},J\text{-}%
VC\right) .$ First of all, we show $\mathcal{F}_{\pi }$ can be made into a
totally bouded pseudo-metric space. Let $d_{\lambda }^{\left( 2\right) }$
and $d_{\nu }^{\left( 2\right) }$ be the $L_{2}\left( \lambda \right) $ and
the $L_{2}\left( \nu \right) $ pseudo-metrics on $\mathcal{H}$ and $\mathcal{%
G}$, respectively. Endow $\mathcal{F}_{\pi }$ with the pseudo-metric%
\begin{equation}
d\left( f_{1},f_{2}\right) \triangleq d_{\lambda }^{\left( 2\right) }\left(
h_{1},h_{2}\right) +d_{\nu }^{\left( 2\right) }\left( g_{1},g_{2}\right) ,
\label{totally-bounded-bis}
\end{equation}%
where $f_{1}=h_{1}g_{1}$ and $f_{2}=h_{2}g_{2}$ are any two elements of $%
\mathcal{F}_{\pi }.$ It suffices to rehearse the arguments in the proof of
Proposition \ref{stoc-bound-ran-covunum-bis} to show that, for all $%
\varepsilon >0,$%
\begin{equation*}
N\left( \varepsilon ,\mathcal{F}_{\pi },d\right) \leq N\left( \varepsilon /2,%
\mathcal{H},d_{\lambda }^{\left( 2\right) }\right) N\left( \varepsilon /2,%
\mathcal{G},d_{\nu }^{\left( 2\right) }\right) <+\infty ,
\end{equation*}%
and, therefore, that $\left( \mathcal{F}_{\pi },d\right) $ is totally
bounded. This fact and the one presented in the following proposition are
key to proving uniform central limit theorems for sequences of $\mathcal{F}%
_{\pi }$-indexed sequential empirical measure processes.

\begin{proposition}
\label{fluctuations-bis}Let $\left( X_{n}\right) _{n\geq 1}$ be a sequence
of i.i.d. random variables defined on a probability space $\left( \Omega ,%
\mathcal{A},P\right) $ and taking values in some measurable space $\left( U,%
\mathcal{U}\right) $. Let $\nu $ be the law of $X_{1}$ and let $\mathcal{F}%
_{\pi }$ be a class of product functions $f=hg,$ where $h\in \mathcal{H}%
\subset \mathfrak{R}^{\left[ 0,1\right] }$ and $g\in \mathcal{G}\subset
\mathfrak{R}^{U}.$ Suppose $\mathcal{H}$ and $\mathcal{G}$ verify conditions
(H1) and (H3), respectively. Let $\left\Vert \cdot \right\Vert _{\left(
\lambda _{n}\otimes \nu \right) }$ be the semi-norm on $\mathcal{F}_{\pi }$
such that%
\begin{equation*}
\left\Vert f\right\Vert _{\left( \lambda _{n}\otimes \nu \right)
}^{2}=\left( \lambda _{n}\otimes \nu \right) \left( f^{2}\right) ,\qquad
f\in \mathcal{F}_{\pi },
\end{equation*}%
where $\lambda _{n}$ is the discrete uniform measure on $\left[ 0,1\right] ,$
and let $d_{\left( \lambda _{n}\otimes \nu \right) }^{\left( 2\right) }$ be
the pseudo-metric on $\mathcal{F}_{\pi }$ defined by%
\begin{equation*}
d_{\left( \lambda _{n}\otimes \nu \right) }^{\left( 2\right) }\left(
f_{1},f_{2}\right) \triangleq \left\Vert f_{1}-f_{2}\right\Vert _{\left(
\lambda _{n}\otimes \nu \right) },\qquad f_{1},f_{2}\in \mathcal{F}_{\pi }.
\end{equation*}%
If%
\begin{equation}
\lim_{n\rightarrow +\infty }\sup_{h_{1},h_{2}\in \mathcal{H}}\left\vert
\lambda _{n}\left( h_{1}-h_{2}\right) ^{2}-\lambda \left( h_{1}-h_{2}\right)
^{2}\right\vert =0,  \label{fluttuazioni-bis}
\end{equation}%
then%
\begin{equation*}
\lim_{\alpha \downarrow 0}\underset{n\rightarrow +\infty }{\lim \sup }%
\sup_{f_{1},f_{2}\in \mathcal{F}_{\pi }^{\alpha }}d_{\left( \lambda
_{n}\otimes \nu \right) }^{\left( 2\right) }\left( f_{1},f_{2}\right) =0,
\end{equation*}%
where $\mathcal{F}_{\pi }^{\alpha }\triangleq \left\{ f_{1},f_{2}\in
\mathcal{F}_{\pi }:d\left( f_{1},f_{2}\right) \leq \alpha \right\} $ and $d$
is defined in (\ref{totally-bounded-bis}).
\end{proposition}

\begin{proof}
Take $f_{1},g_{1}\in \mathcal{F}_{\pi }$ with $f_{1}=h_{1}g_{1}$ and $%
f_{2}=h_{2}g_{2}.$ Then%
\begin{eqnarray*}
d_{\left( \lambda _{n}\otimes \nu \right) }^{\left( 2\right) }\left(
f_{1},f_{2}\right) &=&\left\Vert h_{1}g_{1}-h_{2}g_{2}\right\Vert _{\left(
\lambda _{n}\otimes \nu \right) } \\
&\leq &\left\Vert h_{1}\right\Vert _{\left( \lambda _{n}\otimes \nu \right)
}\left\Vert g_{1}-g_{2}\right\Vert _{\left( \lambda _{n}\otimes \nu \right)
}+\left\Vert g_{2}\right\Vert _{\left( \lambda _{n}\otimes \nu \right)
}\left\Vert h_{1}-h_{2}\right\Vert _{\left( \lambda _{n}\otimes \nu \right) }
\\
&=&\sqrt{\lambda _{n}\left( h_{1}^{2}\right) \nu \left( \left(
g_{1}-g_{2}\right) ^{2}\right) }+\sqrt{\nu \left( g_{2}^{2}\right) \lambda
_{n}\left( \left( h_{1}-h_{2}\right) ^{2}\right) } \\
&\leq &\nu \left( G^{2}\right) \left[ d_{\nu }^{\left( 2\right) }\left(
g_{1},g_{2}\right) +d_{\lambda _{n}}^{\left( 2\right) }\left(
h_{1},h_{2}\right) \right] ,
\end{eqnarray*}%
where $G$ is the envelope of $\mathcal{G}.$ Then, assuming, without loss of
generality, that $\nu \left( G^{2}\right) =1,$ we have%
\begin{eqnarray*}
\sup_{f_{1},f_{2}\in \mathcal{F}_{\pi }^{\alpha }}d_{\left( \lambda
_{n}\otimes \nu \right) }^{\left( 2\right) }\left( f_{1},f_{2}\right) &\leq
&\sup_{g_{1},g_{2}\in \mathcal{G}_{\nu }^{\alpha }}d_{\nu }^{\left( 2\right)
}\left( g_{1},g_{2}\right) +\sup_{h_{1},h_{2}\in \mathcal{H}_{\lambda
}^{\alpha }}d_{\lambda _{n}}^{\left( 2\right) }\left( h_{1},h_{2}\right) \\
&\leq &\alpha +\sup_{h_{1},h_{2}\in \mathcal{H}_{\lambda }^{\alpha
}}d_{\lambda _{n}}^{\left( 2\right) }\left( h_{1},h_{2}\right) ,
\end{eqnarray*}%
where%
\begin{equation*}
\mathcal{G}_{\nu }^{\alpha }\triangleq \left\{ g_{1},g_{2}\in \mathcal{G}%
:d_{\nu }^{\left( 2\right) }\left( g_{1},g_{2}\right) \leq \alpha \right\} ,
\end{equation*}%
and%
\begin{equation*}
\mathcal{H}_{\lambda }^{\alpha }\triangleq \left\{ h_{1},h_{2}\in \mathcal{H}%
:d_{\lambda }^{\left( 2\right) }\left( h_{1},h_{2}\right) \leq \alpha
\right\} .
\end{equation*}%
Now,%
\begin{eqnarray*}
\sup_{h_{1},h_{2}\in \mathcal{H}_{\lambda }^{\alpha }}d_{\lambda
_{n}}^{\left( 2\right) }\left( h_{1},h_{2}\right) &\leq
&\sup_{h_{1},h_{2}\in \mathcal{H}_{\lambda }^{\alpha }}d_{\lambda }^{\left(
2\right) }\left( h_{1},h_{2}\right) \\
&&\qquad \qquad +\sup_{h_{1},h_{2}\in \mathcal{H}_{\lambda }^{\alpha }}\sqrt{%
\left\vert \lambda _{n}\left( h_{1}-h_{2}\right) ^{2}-\lambda \left(
h_{1}-h_{2}\right) ^{2}\right\vert } \\
&\leq &\alpha +\sup_{h_{1},h_{2}\in \mathcal{H}_{\lambda }}\sqrt{\left\vert
\lambda _{n}\left( h_{1}-h_{2}\right) ^{2}-\lambda \left( h_{1}-h_{2}\right)
^{2}\right\vert }.
\end{eqnarray*}%
Taking limits, one has%
\begin{equation*}
\lim_{n\rightarrow +\infty }\sup_{h_{1},h_{2}\in \mathcal{H}_{\lambda
}^{\alpha }}d_{\lambda _{n}}^{\left( 2\right) }\left( h_{1},h_{2}\right)
\leq \alpha ,
\end{equation*}%
so that%
\begin{equation*}
\lim_{n\rightarrow +\infty }\sup_{f_{1},f_{2}\in \mathcal{F}_{\pi }^{\alpha
}}d_{\left( \lambda _{n}\otimes \nu \right) }^{\left( 2\right) }\left(
f_{1},f_{2}\right) \leq 2\alpha .
\end{equation*}%
Taking limits for $\alpha \downarrow 0$ completes the proof.
\end{proof}

\subsection{Proof of Theorem \protect\ref{glivenko-cantelli}.\label{theorem5}%
}

Here is the announced proof of the Glivenko-Cantelli type result for the $B$%
-indexed empirical measure

\begin{lemma}
\label{glivencuccio}Let $\mathcal{W\subset U}$ be a Vapnik-\v{C}ervonenkis
class of measurable subsets of $U.$ Then, for any regular Borel subset, $B,$
of $\left[ 0,1\right] ,$ we have that%
\begin{equation*}
P\left( \lim_{n\rightarrow +\infty }\sup_{W\in \mathcal{W}}\left\vert \nu
_{n,B}\left( W\right) -\nu \left( W\right) \right\vert =0\right) =1,
\end{equation*}%
where $\nu _{n,B}$ is defined in \ref{B-indexed-em}.
\end{lemma}

\begin{proof}
The proof follows closely the lines of the arguments presented in Pollard $%
\left[ 1984\right] ,$ Section II.3, pages 13-16. Almost sure convergence
follows by Borel-Cantelli Lemma, once we show that, for all $\varepsilon >0,$%
\begin{equation}
\sum_{n\geq 1}P\left( \sup_{W\in \mathcal{W}}\left\vert \nu _{n,B}\left(
W\right) -\nu \left( W\right) \right\vert >\varepsilon \right) <+\infty .
\label{borellino-bis}
\end{equation}%
Define the sequence of natural numbers $\left( k_{n}^{B}\right) _{n\geq 1}$
by%
\begin{equation*}
k_{n}^{B}\triangleq card\left( B\cap \left\{ \frac{1}{n},\frac{1}{n-1}%
,...,1\right\} \right) ,
\end{equation*}%
so that the $B$-empirical measure can be written%
\begin{equation*}
\nu _{n,B}=\frac{1}{k_{n}^{B}}\sum_{i\in B\cap \left\{ \frac{1}{n}%
,...,1\right\} }\delta _{X_{i}}.
\end{equation*}%
As in Pollard $\left[ 1984\right] ,$ it is possible to show that for all $%
n\geq 1$ such that $k_{n}^{B}\geq 8\varepsilon ^{-2},$%
\begin{equation*}
P\left( \sup_{W\in \mathcal{W}}\left\vert \nu _{n,B}\left( W\right) -\nu
\left( W\right) \right\vert >\varepsilon \right) \leq 4P\left( \sup_{W\in
\mathcal{W}}\left\vert \sum_{i\in B\cap \left\{ \frac{1}{n},...,1\right\}
}\sigma _{i}\mathbf{1}_{W}\left( X_{i}\right) \right\vert >\frac{\varepsilon
k_{n}^{B}}{4}\right) ,
\end{equation*}%
where $\left( \sigma _{n}\right) _{n\geq 1}$ is a sequence of Rademacher
random variables independent of the sequence $\left( X_{n}\right) _{n\geq 1}.
$ Now we deal with the conditional probability%
\begin{equation*}
P\left( \sup_{W\in \mathcal{W}}\left\vert \sum_{i\in B\cap \left\{ \frac{1}{n%
},...,1\right\} }\sigma _{i}\mathbf{1}_{W}\left( X_{i}\right) \right\vert >%
\frac{\varepsilon k_{n}^{B}}{4}\text{ }|\text{ }\left\{ X_{i}=x_{i}:\frac{i}{%
n}\in B\right\} \right) .
\end{equation*}%
Since $\mathcal{W}$ is a V.C.C., once the $x_{i}$'s are fixed, there exist
sets $W_{1},...,W_{K_{n}^{B}},$ where $K_{n}^{B}$ coincides with the $%
k_{n}^{B}$-th shatter coefficient of $\mathcal{W},$ not depending on the
finite family $\left\{ x_{i}:\frac{i}{n}\in B\right\} $ such that%
\begin{multline*}
P\left( \sup_{W\in \mathcal{W}}\left\vert \sum_{i\in B\cap \left\{ \frac{1}{n%
},...,1\right\} }\sigma _{i}\mathbf{1}_{W}\left( X_{i}\right) \right\vert >%
\frac{\varepsilon k_{n}^{B}}{4}\text{ }|\text{ }\left\{ X_{i}=x_{i}:\frac{i}{%
n}\in B\right\} \right)  \\
=P\left( \max_{1\leq j\leq K_{n}}\left\vert \sum_{i\in B\cap \left\{ \frac{1%
}{n},...,1\right\} }\sigma _{i}\mathbf{1}_{W_{j}}\left( X_{i}\right)
\right\vert >\frac{\varepsilon k_{n}^{B}}{4}\text{ }|\text{ }\left\{
X_{i}=x_{i}:\frac{i}{n}\in B\right\} \right) .
\end{multline*}%
Again, as in Pollard $\left[ 1984\right] ,$ apply the union bound followed
by Hoeffding inequality then integrate out to obtain that, for all $n\geq 1$
such that $k_{n}^{B}\geq 8\varepsilon ^{-2},$%
\begin{equation}
P\left( \sup_{W\in \mathcal{W}}\left\vert \nu _{n,B}\left( W\right) -\nu
\left( W\right) \right\vert >\varepsilon \right) \leq 8K_{n}^{B}\exp \left( -%
\frac{\varepsilon ^{2}k_{n}^{B}}{32}\right) .  \label{piccolino-bis}
\end{equation}%
Since $\mathcal{W}$ is a V.C.C., then $K_{n}^{B}\leq \left(
k_{n}^{B}+1\right) ^{\mathcal{S}\left( \mathcal{W}\right) },$ where $%
\mathcal{S}\left( \mathcal{W}\right) $ is the V.C. dimension of $\mathcal{W}.
$ Combining (\ref{borellino-bis}) and (\ref{piccolino-bis}), it is enough to
show that%
\begin{equation}
\sum_{n\geq 1}\left( k_{n}^{B}+1\right) ^{\mathcal{S}\left( \mathcal{W}%
\right) }\exp \left( -\frac{\varepsilon ^{2}k_{n}^{B}}{32}\right) <+\infty .
\label{massimizzazione-bis}
\end{equation}%
Given $B$ is a regular Borel set, then, by (\ref{dedecker}), $\lim_{n}\frac{%
k_{n}^{B}}{n}=\lambda \left( B\right) ,$ so that, for sufficiently large $n,$
there exists $0<\gamma <\lambda \left( B\right) $ such that%
\begin{equation*}
n\left( \lambda \left( B\right) -\gamma \right) \leq k_{n}^{B}\leq n\left(
\lambda \left( B\right) +\gamma \right) .
\end{equation*}%
It follows that%
\begin{equation*}
\sum_{n\geq 1}\left( k_{n}^{B}+1\right) ^{\mathcal{S}\left( \mathcal{W}%
\right) }\exp \left( -\frac{\varepsilon ^{2}k_{n}^{B}}{32}\right) \leq
\sum_{n\geq 1}\left( n\left( \lambda \left( B\right) +\gamma \right)
+1\right) ^{\mathcal{S}\left( \mathcal{W}\right) }\exp \left( -\frac{n\left(
\lambda \left( B\right) -\gamma \right) \varepsilon ^{2}}{32}\right)
<+\infty ,
\end{equation*}%
since $\lambda \left( B\right) -\gamma >0.$
\end{proof}

\bigskip

\noindent We are now ready to prove Theorem \ref{glivenko-cantelli}. To
prove (\ref{gl-cant-lambaenne}), it is enough to show that, for all $b>0,$%
\begin{equation}
\sum_{n\geq 1}P\left( \sup_{B\in \mathcal{B}_{\#}}\lambda _{n}\left(
B\right) \sup_{W\in \mathcal{W}}\left\vert \mathbb{\nu }_{n,B}\left(
W\right) -\nu \left( W\right) \right\vert >b\right) <+\infty .
\label{borellone-bis}
\end{equation}%
As before, we need to find a suitable bound for%
\begin{multline*}
P\left( \sup_{B\in \mathcal{B}_{\#}}\lambda _{n}\left( B\right) \sup_{W\in
\mathcal{W}}\left\vert \mathbb{\nu }_{n,B}\left( W\right) -\nu \left(
W\right) \right\vert >b\right)  \\
=P\left( \sup_{B\in \mathcal{B}_{\#}}\frac{k_{n}^{B}}{n}\sup_{W\in \mathcal{W%
}}\left\vert \frac{1}{k_{n}^{B}}\sum_{i\in B\cap \left\{ \frac{1}{n}%
,...,1\right\} }\delta _{X_{i}}\left( W\right) -\nu \left( W\right)
\right\vert >b\right) .
\end{multline*}%
Let $\mathcal{J}_{n}$ be the class of all subsets of $\left\{ \frac{1}{n}%
,...,1\right\} .$ Write%
\begin{equation*}
\mathcal{J}_{n}=\bigcup\limits_{k=0}^{n}\mathcal{J}_{n}^{k},
\end{equation*}%
where $\mathcal{J}_{n}^{k}\triangleq \left\{ J_{n,j}^{k},\text{ }j=1,...,%
\binom{n}{k}\right\} ,$ $k=0,...,n,$ is the class of all subsets of $\left\{
\frac{1}{n},...,1\right\} $ of cardinality exactly equal to $k.$ Define%
\begin{equation*}
\mathcal{J}_{n}\left( \mathcal{B}_{\#}\right) \triangleq \left\{ B\cap
\left\{ \frac{1}{n},...,1\right\} :B\in \mathcal{B}_{\#}\right\} .
\end{equation*}%
Since $\mathcal{B}_{\#}$ is a V.C.C., $card\left( \mathcal{J}_{n}\left(
\mathcal{B}_{\#}\right) \right) \leq m^{\mathcal{B}_{\#}}\left( n\right) ,$
where $m^{\mathcal{B}_{\#}}\left( n\right) $ is the $n$-th shatter
coefficient of $\mathcal{B}_{\#}.$ Finally, define%
\begin{equation*}
\mathcal{J}_{n}^{k}\left( \mathcal{B}_{\#}\right) \triangleq \mathcal{J}%
_{n}^{k}\cap \mathcal{J}_{n}\left( \mathcal{B}_{\#}\right) ,
\end{equation*}%
and let%
\begin{equation*}
K_{n,K}^{\mathcal{B}_{\#}}\triangleq card\left( \mathcal{J}_{n}^{k}\left(
\mathcal{B}_{\#}\right) \right) .
\end{equation*}%
Clearly,%
\begin{equation*}
\sum_{k=0}^{n}K_{n,k}^{\mathcal{B}_{\#}}\leq m^{\mathcal{B}_{\#}}\left(
n\right) .
\end{equation*}%
It follows that%
\begin{multline*}
P\left( \sup_{B\in \mathcal{B}_{\#}}\frac{k_{n}^{B}}{n}\sup_{W\in \mathcal{W}%
}\left\vert \frac{1}{k_{n}^{B}}\sum_{i\in B\cap \left\{ \frac{1}{n}%
,...,1\right\} }\delta _{X_{i}}\left( W\right) -\nu \left( W\right)
\right\vert >b\right)  \\
\leq P\left( \max_{0\leq k\leq n}\frac{k}{n}\max_{J\in \mathcal{J}%
_{n}^{k}\left( \mathcal{B}_{\#}\right) }\sup_{W\in \mathcal{W}}\left\vert
\nu _{J}\left( W\right) -\nu \left( W\right) \right\vert >b\right) ,
\end{multline*}%
where $\nu _{J}\triangleq \frac{1}{card\left( J\right) }\sum_{i\in J\cap
\left\{ \frac{1}{n},...,1\right\} }\delta _{X_{i}},$ with the convention
that $\nu _{J}\equiv 0$ if $J=\varnothing .$ Then,%
\begin{equation*}
P\left( \max_{0\leq k\leq n}\frac{k}{n}\max_{J\in \mathcal{J}_{n}^{k}\left(
\mathcal{B}_{\#}\right) }\sup_{W\in \mathcal{W}}\left\vert \nu _{J}\left(
W\right) -\nu \left( W\right) \right\vert >b\right) =P\left( \max_{1\leq
k\leq n}\max_{J\in \mathcal{J}_{n}^{k}\left( \mathcal{B}_{\#}\right)
}\sup_{W\in \mathcal{W}}\left\vert \nu _{J}\left( W\right) -\nu \left(
W\right) \right\vert >\frac{bn}{k}\right) .
\end{equation*}%
Now,%
\begin{eqnarray*}
&&P\left( \max_{1\leq k\leq n}\max_{J\in \mathcal{J}_{n}^{k}\left( \mathcal{B%
}_{\#}\right) }\sup_{W\in \mathcal{W}}\left\vert \nu _{J}\left( W\right)
-\nu \left( W\right) \right\vert >\frac{bn}{k}\right)  \\
&\leq &\sum_{k=1}^{n}\sum_{J\in \mathcal{J}_{n}^{k}\left( \mathcal{B}%
_{\#}\right) }P\left( \sup_{W\in \mathcal{W}}\left\vert \nu _{J}\left(
W\right) -\nu \left( W\right) \right\vert >\frac{bn}{k}\right)  \\
&\leq &\sum_{k=1}^{n}K_{n,k}^{\mathcal{B}_{\#}}\max_{J\in \mathcal{J}%
_{n}^{k}\left( \mathcal{B}_{\#}\right) }P\left( \sup_{W\in \mathcal{W}%
}\left\vert \nu _{J}\left( W\right) -\nu \left( W\right) \right\vert >\frac{%
bn}{k}\right)  \\
&\leq &\sum_{k=1}^{n}m^{\mathcal{B}_{\#}}\left( n\right) P\left( \sup_{W\in
\mathcal{W}}\left\vert \nu _{k}\left( W\right) -\nu \left( W\right)
\right\vert >\frac{bn}{k}\right) ,
\end{eqnarray*}%
where $\nu _{k}\triangleq k^{-1}\sum_{i=1}^{k}\delta _{X_{i}},$ and where
the last inequality follows from independence and identity in distribution
of the $X_{i}$'s. Now, apply (\ref{piccolino-bis}) for $B=\left[ 0,1\right] $
to obtain%
\begin{equation*}
P\left( \sup_{W\in \mathcal{W}}\left\vert \nu _{k}\left( W\right) -\nu
\left( W\right) \right\vert >\frac{bn}{k}\right) \leq 8m^{\mathcal{W}}\left(
k\right) \exp \left( -\frac{b^{2}n^{2}}{32k}\right) ,
\end{equation*}%
where $m^{\mathcal{W}}\left( k\right) $ is the $k$-th shatter coefficient of
$\mathcal{W}.$ Consequently,%
\begin{equation*}
P\left( \max_{1\leq k\leq n}\max_{J\in \mathcal{J}_{n}^{k}\left( \mathcal{B}%
_{\#}\right) }\sup_{W\in \mathcal{W}}\left\vert \nu _{J}\left( W\right) -\nu
\left( W\right) \right\vert >\frac{bn}{k}\right) \leq 8\sum_{k=1}^{n}m^{%
\mathcal{B}_{\#}}\left( n\right) m^{\mathcal{W}}\left( k\right) \exp \left( -%
\frac{b^{2}n^{2}}{32k}\right) .
\end{equation*}%
To show (\ref{borellone-bis}), it is therefore enough to show that%
\begin{equation}
\sum_{n\geq 1}\sum_{k=1}^{n}m^{\mathcal{B}_{\#}}\left( n\right) m^{\mathcal{W%
}}\left( k\right) \exp \left( -\frac{b^{2}n^{2}}{32k}\right) =\sum_{k\geq
1}m^{\mathcal{W}}\left( k\right) \sum_{n\geq k}m^{\mathcal{B}_{\#}}\left(
n\right) \exp \left( -\frac{b^{2}n^{2}}{32k}\right) <+\infty .
\label{borellissimo-bis}
\end{equation}%
Since $n\geq k,$ we have $\exp \left( -\frac{b^{2}n^{2}}{32k}\right) \leq
\exp \left( -\frac{b^{2}n}{32}\right) .$ Moreover, for all $n\geq k>\max
\left\{ \mathcal{S}\left( \mathcal{B}_{\#}\right) ,\mathcal{S}\left(
\mathcal{W}\right) \right\} ,$ where $\mathcal{S}\left( \mathcal{B}%
_{\#}\right) $ and $\mathcal{S}\left( \mathcal{W}\right) $ denote the V.C.
dimensions of $\mathcal{B}_{\#}$ and $\mathcal{W}$, respectively, we have
that $m^{\mathcal{W}}\left( k\right) \leq \left( k+1\right) ^{\mathcal{S}%
\left( \mathcal{W}\right) }<M\left( \mathcal{W}\right) k^{\mathcal{S}\left(
\mathcal{W}\right) }$ and that $m^{\mathcal{B}_{\#}}\left( n\right) \leq
\left( n+1\right) ^{\mathcal{S}\left( \mathcal{B}_{\#}\right) }<M\left(
\mathcal{B}_{\#}\right) n^{\mathcal{S}\left( \mathcal{B}_{\#}\right) },$ $%
M\left( \mathcal{W}\right) $ and $M\left( \mathcal{B}_{\#}\right) $ being
constants depending only on classes $\mathcal{W}$ and $\mathcal{B}_{\#},$
respectively. The convergence of the double series in the RHS of (\ref%
{borellissimo-bis}) follows if it holds that, for all $c>0,$%
\begin{equation*}
\sum_{k\geq 1}k^{\mathcal{S}\left( \mathcal{W}\right) }\sum_{n\geq k}n^{%
\mathcal{S}\left( \mathcal{B}_{\#}\right) }\exp \left( -cn\right) <+\infty ,
\end{equation*}%
which is, in turn, true if for all $c>0$ and all $D_{1},D_{2}\in \mathbb{N}%
^{+},$%
\begin{equation*}
I\left( c,D_{1},D_{2}\right) \triangleq \int_{1}^{\infty }\int_{y}^{+\infty
}y^{D_{1}}x^{D_{2}}\exp \left( -cx\right) \text{d}x\text{d}y<+\infty .
\end{equation*}%
Elementary calculus gives%
\begin{equation}
I\left( c,D_{1},D_{2}\right)
=e^{-c}\sum_{p=0}^{D_{2}}\sum_{l=0}^{D_{1}+D_{2}-p}c^{-\left( p+l+2\right) }%
\frac{D_{2}!\left( D_{1}+D_{2}-p\right) !}{\left( D_{2}-p\right) !\left(
D_{1}+D_{2}-p-l\right) !}<+\infty ,  \label{elementary-bis}
\end{equation}%
completing the proof of (\ref{gl-cant-lambaenne}). As for (\ref{gl-cant}),
just observe that%
\begin{equation*}
\sup_{B\in \mathcal{B}_{\#}}\sup_{W\in \mathcal{W}}\left\vert \mathbb{P}%
_{n}\left( B\times W\right) -\lambda \left( B\right) \nu \left( W\right)
\right\vert \leq \sup_{B\in \mathcal{B}_{\#}}\lambda _{n}\left( B\right)
\sup_{W\in \mathcal{W}}\left\vert \nu _{n,B}\left( W\right) -\nu \left(
W\right) \right\vert +\sup_{B\in \mathcal{B}_{\#}}\left\vert \lambda
_{n}\left( B\right) -\lambda \left( B\right) \right\vert .
\end{equation*}%
The first term of the RHS converges almost surely to $0$ thanks to (\ref%
{elementary-bis}) so that the deterministic convergence to $0$ described in (%
\ref{dimenticata}) implies (\ref{gl-cant}).

\subsection{Functional central limit theorems\label{proof-fclt's}}

Let $\mathcal{Q}\subset \mathfrak{R}^{\left[ 0,1\right] \times U}$ be a
class of $\mathcal{B}\left( \left[ 0,1\right] \otimes \mathcal{U}\right) $%
-measurable functions defined by conditions (a) and (b) of Section \ref%
{fclt's}. Let $\left( X_{n}\right) _{n\geq 1}$ be a sequence of i.i.d.
random variables of law $\nu $ defined on a probability space $\left( \Omega
,\mathcal{A},P\right) $ and taking values in some measurable space $\left( U,%
\mathcal{U}\right) .$ As already announced in that Section, will will prove
convergence of finite dimensional laws of the sequence of $\mathcal{Q}$%
-indexed processes $\left\{ Z_{n}\left( q\right) :q\in \mathcal{Q}\right\} ,$
$n\geq 1,$ where%
\begin{equation*}
Z_{n}\left( q\right) \triangleq \sqrt{n}\left( \mathbb{P}_{n}\left( q\right)
-\left( \lambda _{n}\otimes \nu \right) \left( q\right) \right) ,
\end{equation*}%
$\lambda _{n}$ being the discrete uniform measure and $\mathbb{P}_{n}\left(
q\right) $ being defined in (\ref{iniziale}), to those of the $\mathcal{Q}$%
-indexed centered Gaussian process $\left\{ Z\left( q\right) :q\in \mathcal{Q%
}\right\} $ whose covariance structure is given by (\ref{kieferissimo}). To
begin with, we will prove that, for all $q\in \mathcal{Q},$ $Z_{n}\left(
q\right) \overset{\mathcal{L}}{\rightarrow }Z\left( q\right) .$We first need
to make some remarks and introduce some notation. Note that if $q\in
\mathcal{Q}$ then $\tilde{q}\in \mathcal{Q}$ where $\tilde{q}$ is defined,
for all $\left( s,x\right) \in \left[ 0,1\right] \times U,$ by%
\begin{equation}
\tilde{q}\left( s,x\right) \triangleq q\left( s,x\right) -\nu \left(
q\right) \left( s\right) .  \label{q-tilde}
\end{equation}%
Now, use the sequence $\left( X_{n}\right) _{n\geq 1}$ to construct the
triangular array $\left( X_{in}\right) _{1\leq i\leq n,n\geq 1}$ of
independent random variables as follows:%
\begin{equation*}
X_{in}\triangleq \frac{\tilde{q}\left( i/n,X_{i}\right) }{\sqrt{n}},
\end{equation*}%
where $\tilde{q}$ is defined in (\ref{q-tilde}). Note that%
\begin{equation*}
Z_{n}\left( q\right) =\sum_{i=1}^{n}X_{in},
\end{equation*}%
and that
\begin{equation*}
Z\left( q\right) \overset{d}{=}N\left( 0,\left( \lambda \otimes \nu \right)
\left( \tilde{q}^{2}\right) \right) .
\end{equation*}

\begin{lemma}
\label{lindebergone}For all $q\in \mathcal{Q},$ $Z_{n}\left( q\right)
\overset{\mathcal{L}}{\rightarrow }Z\left( q\right) .$
\end{lemma}

\begin{proof}
We first treat the case when $\left( \lambda \otimes \nu \right) \left(
\tilde{q}^{2}\right) =0.$ In this case, $Z\left( q\right) \overset{d}{=}%
\delta _{0}.$ Moreover, it is easily seen that $E\left( \left( Z_{n}\left(
q\right) \right) ^{2}\right) =\left( \lambda _{n}\otimes \nu \right) \left(
\tilde{q}^{2}\right) ,$so that, by (\ref{namely}), $Z_{n}\left( q\right)
\overset{L_{2}}{\rightarrow }Z\left( q\right) ,$and consequently $%
Z_{n}\left( q\right) \overset{\mathcal{L}}{\rightarrow }\delta _{0}.$ In the
case in which $\left( \lambda \otimes \nu \right) \left( \tilde{q}%
^{2}\right) >0,$ Lindeberg central limit theorem implies that $Z_{n}\left(
q\right) \overset{\mathcal{L}}{\rightarrow }Z\left( q\right) $ if%
\begin{equation}
\forall \varepsilon >0,\qquad \lim_{n\rightarrow +\infty }\frac{1}{n\left(
\lambda _{n}\otimes \nu \right) \left( \tilde{q}^{2}\right) }%
\sum_{i=1}^{n}\int_{\left\{ x:\left\vert \tilde{q}\left( i/n,X_{i}\right)
\right\vert \geq \varepsilon \sqrt{n\left( \lambda _{n}\otimes \nu \right)
\left( \tilde{q}^{2}\right) }\right\} }\tilde{q}^{2}\left( i/n,X_{i}\right)
\nu \left( \text{d}x\right) =0.  \label{cond-lindeberg}
\end{equation}%
By condition (a) in section \ref{fclt's}, by (\ref{namely}) and by the fact
that, for all $\varepsilon >0,$%
\begin{equation*}
\left\{ x:\left\vert \tilde{q}\left( i/n,X_{i}\right) \right\vert \geq
\varepsilon \sqrt{n\left( \lambda _{n}\otimes \nu \right) \left( \tilde{q}%
^{2}\right) }\right\} \subset \left\{ x:\left\vert g_{\tilde{q}}\left(
x\right) \right\vert \geq \varepsilon \sqrt{n\left( \lambda _{n}\otimes \nu
\right) \left( \tilde{q}^{2}\right) }\right\} ,
\end{equation*}%
for sufficiently large $n,$ there exists $0<\gamma <\left( \lambda \otimes
\nu \right) \left( \tilde{q}^{2}\right) $ such that, for all $\varepsilon
>0, $%
\begin{multline*}
\frac{1}{n\left( \lambda _{n}\otimes \nu \right) \left( \tilde{q}^{2}\right)
}\sum_{i=1}^{n}\int_{\left\{ x:\left\vert \tilde{q}\left( i/n,X_{i}\right)
\right\vert \geq \varepsilon \sqrt{n\left( \lambda _{n}\otimes \nu \right)
\left( \tilde{q}^{2}\right) }\right\} }\tilde{q}^{2}\left( i/n,X_{i}\right)
\nu \left( \text{d}x\right) \\
\leq \frac{1}{\left( \lambda \otimes \nu \right) \left( \tilde{q}^{2}\right)
}\int_{\left\{ x:\left\vert g_{\tilde{q}}\left( x\right) \right\vert \geq
\varepsilon \sqrt{n\left( \lambda _{n}\otimes \nu \right) \left( \tilde{q}%
^{2}\right) }\right\} }g_{\tilde{q}}^{2}\left( x\right) \nu \left( \text{d}%
x\right) .
\end{multline*}%
Since $g_{\tilde{q}}\in L_{2}\left( \nu \right) ,$%
\begin{equation*}
\lim_{n\rightarrow +\infty }\int_{\left\{ x:\left\vert g_{\tilde{q}}\left(
x\right) \right\vert \geq \varepsilon \sqrt{n\left( \lambda _{n}\otimes \nu
\right) \left( \tilde{q}^{2}\right) }\right\} }g_{\tilde{q}}^{2}\left(
x\right) \nu \left( \text{d}x\right) =0,
\end{equation*}%
which implies (\ref{cond-lindeberg}).
\end{proof}

\bigskip

\noindent To complete the proof of Proposition \ref{finitodimensionali}, we
will employ the Cram\'{e}r-Wold device. Fix a natural number $1\leq
K<+\infty $ and functions $g_{1},...,g_{K}\in \mathcal{Q}$. Consider the
random vector $Z_{nK}=\left( Z_{n}\left( q_{1}\right) ,...,Z_{n}\left(
q_{K}\right) \right) .$ To show that the sequence $\left( Z_{nK}\right)
_{n\geq 1}$ converges in law to the vector $Z_{nK}=\left( Z\left(
q_{1}\right) ,...,Z\left( q_{K}\right) \right) ,$ it is enough to show that
the sequence of random vectors $\tilde{Z}_{nK}=\left( Z_{n}\left( \tilde{q}%
_{1}\right) ,...,Z\left( \tilde{q}_{K}\right) \right) $, $n\geq 1,$
converges in law to the random vector $\tilde{Z}_{nK}=\left( Z\left( \tilde{q%
}_{1}\right) ,...,Z\left( \tilde{q}_{K}\right) \right) .$ This is equivalent
to showing that, for all $\left( a_{1},...,a_{K}\right) \in \mathfrak{R}^{K},
$%
\begin{equation*}
\sum_{i=1}^{K}a_{i}Z_{n}\left( \tilde{q}_{i}\right) \overset{\mathcal{L}}{%
\rightarrow }\sum_{i=1}^{K}a_{i}Z\left( \tilde{q}_{i}\right) .
\end{equation*}%
It is easily seen that%
\begin{equation*}
\sum_{i=1}^{K}a_{i}Z_{n}\left( \tilde{q}_{i}\right) =Z_{n}\left(
\sum_{i=1}^{K}a_{i}\tilde{q}_{i}\right) ,
\end{equation*}%
and, since $\mathcal{Q}$ is a linear space, we have by Lemma \ref%
{lindebergone},%
\begin{equation*}
Z_{n}\left( \sum_{i=1}^{K}a_{i}\tilde{q}_{i}\right) \overset{\mathcal{L}}{%
\rightarrow }Z\left( \sum_{i=1}^{K}a_{i}\tilde{q}_{i}\right) .
\end{equation*}%
Finally, it is easily seen that%
\begin{equation*}
Var\left( Z\left( \sum_{i=1}^{K}a_{i}\tilde{q}_{i}\right) \right) =Var\left(
\sum_{i=1}^{K}a_{i}Z\left( \tilde{q}_{i}\right) \right) =\left( \lambda
\otimes \nu \right) \left( \left( \sum_{i=1}^{K}a_{i}\tilde{q}_{i}\right)
^{2}\right) ,
\end{equation*}%
proving Proposition \ref{finitodimensionali}.

\bigskip

\noindent Let's return back to $\mathcal{F}_{\pi }$-indexed s.e.m.p.'s with $%
\mathcal{F}_{\pi }\in \pi \left( \nu G^{2},J\text{-}VC\right) $ or $\mathcal{%
F}_{\pi }\in \pi \left( UB,M\text{-}VC\right) .$ To prove weak convergence
in the former case, we will use the findings of section 4.2 in Ziegler $%
\left[ 1997\right] .$ The following Theorem summarizes the results needed.

\begin{theorem}
\label{ziegler-97-par-4.2}\textbf{(Ziegler, 1997, Paragraph 4.2)} Let $%
\left( \Upsilon ,\mathcal{E}\right) $ be a measurable space and let $\left(
\eta _{ni}\right) _{1\leq i\leq i\left( n\right) ,n\geq 1}$ be a triangular
array of rowwise independent $\Upsilon $-valued random variables with laws $%
\nu _{ni},$ respectively. Let $\mathcal{F}$ be a class of $\mathcal{E}$%
-measurable real valued functions defined on $\Upsilon $ with envelope $F$
and assume that $\mathcal{F}$ has a uniformly integrable $L_{2}$-entropy
(see Ziegler $\left[ 1997\right] $ for the definition of uniformly
integrable $L_{2}$-entropy of a class $\mathcal{F}$). Assume also that there
exists a metric $d$ on $\mathcal{F}$ such that $\left( \mathcal{F},d\right) $
is a totally bounded metric space. Consider the $\mathcal{F}$-indexed
stochastic processes $\left\{ S_{n}\left( f\right) :f\in \mathcal{F}\right\}
,$ $n\geq 1,$ where, for all $n\geq 1,$%
\begin{equation*}
S_{n}\left( f\right) \triangleq \frac{1}{\sqrt{i\left( n\right) }}%
\sum_{1\leq i\leq i\left( n\right) }\left( f\left( \eta _{ni}\right) -\nu
_{ni}\left( f\right) \right) ,\qquad f\in \mathcal{F}.
\end{equation*}%
define the probability measure%
\begin{equation*}
\tilde{\nu}_{n}\triangleq \frac{1}{i\left( n\right) }\sum_{1\leq i\leq
i\left( n\right) }\nu _{ni},
\end{equation*}%
and the quantity%
\begin{equation*}
a_{n}\left( \alpha \right) \triangleq \sup_{\left\{ f,g\in \mathcal{F}%
:d\left( f,g\right) \leq \alpha \right\} }\sqrt{\tilde{\nu}_{n}\left( \left(
f-g\right) ^{2}\right) },\qquad \alpha >0.
\end{equation*}%
If there exists a centered $\mathcal{F}$-indexed Gaussian process $\bar{G}%
=\left\{ G\left( f\right) :f\in \mathcal{F}\right\} $ such that the finite
dimensional distributions of the sequence of processes $S_{n}$ converge to
those of $\bar{G}$ and if

\begin{enumerate}
\item[(i)] $\sup_{n\geq 1}\tilde{\nu}_{n}\left( F^{2}\right) <+\infty ;$

\item[(ii)] $\lim_{\alpha \downarrow 0}\underset{n\rightarrow +\infty }{\lim
\sup \text{ }}a_{n}\left( \alpha \right) =0;$

\item[(iii)] for all $\delta >0,$%
\begin{equation*}
\lim_{n\rightarrow +\infty }\sum_{1\leq i\leq i\left( n\right) }E\left(
F^{2}\left( \eta _{ni}\right) \cdot \mathbf{1}_{\left( \delta \sqrt{i\left(
n\right) },+\infty \right) }\left( F\left( \eta _{ni}\right) \right) \right)
=0,
\end{equation*}%
then there exists a version $G$ of $\bar{G},$ with uniformly bounded and
uniformly continuous sample paths such that%
\begin{equation*}
S_{n}\underset{sep}{\overset{\mathcal{L}}{\rightarrow }}G.
\end{equation*}
\end{enumerate}
\end{theorem}

\bigskip

\noindent Theorem \ref{fclt_nuggi} is a corollary of Theorem \ref%
{ziegler-97-par-4.2}. To see this, define the triangular array of rowwise
independent random variables%
\begin{equation*}
\eta _{ni}=\left( i/n,X_{i}\right) ,\qquad 1\leq i\leq n,\text{ }n\leq 1,
\end{equation*}%
and observe that, for all $1\leq i\leq n$ and all $n\geq 1,$ the law of $%
\eta _{ni}$ is $\delta _{i/n}\otimes \nu .$ The sequence of stochastic
processes $\left\{ S_{n}\left( f\right) :f\in \mathcal{F}_{\pi }\right\} $
described in the statement Theorem \ref{ziegler-97-par-4.2} is nothing but
the sequence of stochastic processes $\left\{ Z_{n}\left( f\right) :f\in
\mathcal{F}_{\pi }\right\} .$ Now, $\left( \mathcal{F}_{\pi },d\right) $,
with $d=d_{\lambda }^{\left( 2\right) }+d_{\nu }^{\left( 2\right) }$ is a
totally bouded and since any $\mathcal{F}_{\pi }\in \pi \left( \nu G^{2},J%
\text{-}VC\right) $ is by definition a V.C.G.C., then it posseses a
uniformly integrable $L_{2}$-entropy. As for convergence of finite
dimensional laws, since $\mathcal{F}_{\pi }\subset \mathcal{Q},$ then the
finite dimensional laws of the process $Z_{n}$ converge to those of $Z.$ We
are, therefore, only left with the verification of conditions (i), (ii) and
(iii) of Theorem \ref{ziegler-97-par-4.2}. Condition (i) is immediate since
it translates as follows:%
\begin{equation*}
\sup_{n\geq 1}\left( \lambda _{n}\otimes \nu \right) \left( G^{2}\right)
<+\infty ,
\end{equation*}%
and $G\in L_{2}\left( \nu \right) .$ Moreover, condition (iii) applied to
our case becomes%
\begin{equation*}
\forall \delta >0,\qquad \frac{1}{n}\sum_{i=1}^{n}E\left( G^{2}\left(
X_{i}\right) \cdot \mathbf{1}_{\left( \delta \sqrt{n},+\infty \right)
}\left( G\left( X_{i}\right) \right) \right) =0,
\end{equation*}%
which is true by dominated convergence, since $G\in L_{2}\left( \nu \right)
. $ Finally, condition (ii) is exactly the conclusion of Proposition \ref%
{fluctuations-bis}.

\bigskip

\noindent We end this section with the proof of Theorem \ref{fclt_ubbi}. All
the other conditions being true as for the case of $\mathcal{F}_{\pi }\in
\pi \left( \nu G^{2},J\text{-}VC\right) ,$ we only need to prove $d$%
-equicontinuity of the sequence $\mathcal{F}_{\pi }$-indexed stochastic
processes $\left\{ Z_{n}\left( f\right) :f\in \mathcal{F}_{\pi }\right\} ,$
when $\mathcal{F}_{\pi }\in \pi \left( UB,M\text{-}VC\right) $. As argued in
Ziegler $\left[ 1997\right] ,$ it is enough to show that%
\begin{equation*}
\lim_{\alpha \downarrow 0}\underset{n\rightarrow +\infty }{\lim \sup }\text{
}E^{\ast }\left( \sup_{\left\{ f_{1},f_{2}\in \mathcal{F}_{\pi }:d\left(
f_{1},f_{2}\right) \leq \alpha \right\} }\left\vert \frac{1}{\sqrt{n}}%
\sum_{i=1}^{n}\varepsilon _{i}\left[ h_{1}\left( i/n\right) g_{1}\left(
X_{i}\right) -h_{2}\left( i/n\right) g_{2}\left( X_{i}\right) \right]
\right\vert \right) =0,
\end{equation*}%
where $E^{\ast }$ denotes the outer expectation operator, where $\left(
\varepsilon _{n}\right) _{n\geq 1}$ is a canonically formed Rademacher
sequence and where we have written $f_{1}=h_{1}g_{1}$ and $f_{2}=h_{2}g_{2}.$
For all $\mathcal{F}_{\pi }\ni f=hg,$ define the triangular array%
\begin{equation*}
\Phi _{ni}\left( f\right) \triangleq \frac{1}{\sqrt{n}}h\left( i/n\right)
g\left( X_{i}\right) ,
\end{equation*}%
the random seminorm $\left\Vert \cdot \right\Vert _{\rho _{n}}$ such that,
for all $f\in \mathcal{F}_{\pi },$%
\begin{equation*}
\left\Vert f\right\Vert _{\rho _{n}}^{2}=\sum_{i=1}^{n}\Phi _{ni}^{2}\left(
f\right) ,
\end{equation*}%
and the random pseudo-metric%
\begin{equation*}
d_{\rho _{n}}^{\left( 2\right) }\left( f_{1},f_{2}\right) \triangleq
\left\Vert f_{1}-f_{2}\right\Vert _{\rho _{n}}.
\end{equation*}%
With this notation, we need to prove that%
\begin{equation}
\lim_{\alpha \downarrow 0}\underset{n\rightarrow +\infty }{\lim \sup }\text{
}E^{\ast }\left( \sup_{\left\{ f_{1},f_{2}\in \mathcal{F}_{\pi }:d\left(
f_{1},f_{2}\right) \leq \alpha \right\} }\left\vert
\sum_{i=1}^{n}\varepsilon _{i}\left[ \Phi _{ni}\left( f_{1}\right) -\Phi
_{ni}\left( f_{2}\right) \right] \right\vert \right) =0.
\label{equicontinuity}
\end{equation}%
Note that%
\begin{equation*}
d_{\left( \lambda _{n}\otimes \nu \right) }^{\left( 2\right) }\left(
f_{1},f_{2}\right) =\left\Vert f_{1}-f_{2}\right\Vert _{\left( \lambda
_{n}\otimes \nu \right) }=\sqrt{E\left( \left\Vert f_{1}-f_{2}\right\Vert
_{\rho _{n}}^{2}\right) },
\end{equation*}%
so that Theorem 3.1 in Ziegler $\left[ 1997\right] $ yields the existence of
universal constants $K_{1}$ and $K_{2}$ such that%
\begin{equation*}
E^{\ast }\left( \sup_{\left\{ f_{1},f_{2}\in \mathcal{F}_{\pi }:d\left(
f_{1},f_{2}\right) \leq \alpha \right\} }\left\vert
\sum_{i=1}^{n}\varepsilon _{i}\left[ \Phi _{ni}\left( f_{1}\right) -\Phi
_{ni}\left( f_{2}\right) \right] \right\vert \right) \leq K_{1}A\left(
n,\alpha \right) B\left( n\right) +K_{2}C\left( n,\alpha \right) ,
\end{equation*}%
with%
\begin{eqnarray*}
A\left( n,\alpha \right) &=&\sqrt{E^{\ast }\left( \max_{1\leq i\leq
n}\sup_{f_{1},f_{2}\in \mathcal{F}_{\pi }}\frac{1}{\sqrt{n}}\left\vert
h_{1}\left( i/n\right) g_{1}\left( X_{i}\right) -h_{2}\left( i/n\right)
g_{2}\left( X_{i}\right) \right\vert \cdot l_{n}\left( 1\right) \right) } \\
&\leq &\frac{\sqrt{2}}{\alpha \sqrt[4]{n}}\sqrt{E^{\ast }\left( l_{n}\left(
1\right) \right) },
\end{eqnarray*}%
\begin{equation*}
B\left( n\right) =\sqrt{E^{\ast }\left( l_{n}\left( 1\right) ^{2}\right) },
\end{equation*}%
\begin{equation*}
C\left( n,\alpha \right) =E\left( l_{n}\left( \alpha \right) \right) ,
\end{equation*}%
and%
\begin{equation*}
l_{n}\left( \alpha \right) =\int_{0}^{\alpha }\sqrt{\log N\left( \tau ,%
\mathcal{F}_{\pi },d_{\rho _{n}}^{\left( 2\right) }\right) }\text{d}\tau
,\qquad \alpha >0.
\end{equation*}%
Clearly, for all $f_{1}=h_{1}g_{1}$ and $f_{2}=h_{2}g_{2}$ in $\mathcal{F}%
_{\pi },$ all $\omega \in \Omega $ and all $n\geq 1,$%
\begin{equation*}
d_{\rho _{n}}^{\left( 2\right) }\left( f_{1},f_{2}\right) \leq d_{\nu
_{n}}^{\left( 2\right) }\left( g_{1},g_{2}\right) +d_{\lambda _{n}}^{\left(
2\right) }\left( h_{1},h_{2}\right) ,
\end{equation*}%
where%
\begin{equation*}
d_{\nu _{n}}^{\left( 2\right) }\left( g_{1},g_{2}\right) \triangleq \sqrt{%
\frac{1}{n}\sum_{i=1}^{n}\left( g_{1}\left( X_{i}\right) -g_{2}\left(
X_{i}\right) \right) ^{2}},
\end{equation*}%
and%
\begin{equation*}
d_{\lambda _{n}}^{\left( 2\right) }\left( h_{1},h_{2}\right) \triangleq
\sqrt{\frac{1}{n}\sum_{i=1}^{n}\left( h_{1}\left( i/n\right) -h_{2}\left(
i/n\right) \right) ^{2}}.
\end{equation*}%
Rehearse the arguments of Proposition \ref{stoc-bound-ran-covunum-bis} to
show that, for all $\omega \in \Omega ,$ all $\tau >0$ and all $n\geq 1,$%
\begin{equation*}
N\left( \tau ,\mathcal{F}_{\pi },d_{\rho _{n}}^{\left( 2\right) }\right)
\leq N\left( \tau ,\mathcal{H},d_{\lambda _{n}}^{\left( 2\right) }\right)
N\left( \tau ,\mathcal{G},d_{\nu _{n}}^{\left( 2\right) }\right) .
\end{equation*}%
It follows that, for all $\alpha >0,$ all $\omega \in \Omega $ and all $%
n\geq 1,$%
\begin{equation*}
\int_{0}^{\alpha }\sqrt{\log N\left( \tau ,\mathcal{F}_{\pi },d_{\rho
_{n}}^{\left( 2\right) }\right) }\text{d}\tau \leq \int_{0}^{\alpha }\sqrt{%
\log N\left( \tau ,\mathcal{H},d_{\lambda _{n}}^{\left( 2\right) }\right) }%
\text{d}\tau +\int_{0}^{\alpha }\sqrt{\log N\left( \tau ,\mathcal{G},d_{\nu
_{n}}^{\left( 2\right) }\right) }\text{d}\tau .
\end{equation*}%
Since, for all $n\geq 1,$ $\lambda _{n}$ is a finite measure on $\left[ 0,1%
\right] $ and since $\mathcal{H}$ is a V.C.G.C., then there exists a
function $\gamma _{\mathcal{H}}\left( \tau \right) $ such that, for all $%
n\geq 1,$%
\begin{equation*}
\int_{0}^{\alpha }\sqrt{\log N\left( \tau ,\mathcal{H},d_{\lambda
_{n}}^{\left( 2\right) }\right) }\text{d}\tau \leq \int_{0}^{\alpha }\sqrt{%
\log \gamma _{\mathcal{H}}\left( \tau \right) }\text{d}\tau <+\infty .
\end{equation*}%
Analogously, there exists a function $\gamma _{\mathcal{G}}\left( \tau
\right) $ such that, for all $n\geq 1$ and all $\omega \in \Omega ,$%
\begin{equation*}
\int_{0}^{\alpha }\sqrt{\log N\left( \tau ,\mathcal{G},d_{\nu _{n}}^{\left(
2\right) }\right) }\text{d}\tau \leq \int_{0}^{\alpha }\sqrt{\log \gamma _{%
\mathcal{G}}\left( \tau \right) }\text{d}\tau <+\infty .
\end{equation*}%
Consequently,%
\begin{equation*}
E\left( l_{n}\left( 1\right) ^{2}\right) <+\infty ,
\end{equation*}%
which in turn implies that, for all $\alpha >0,$%
\begin{equation*}
\underset{n\rightarrow +\infty }{\lim \sup }\text{ }A\left( n,\alpha \right)
B\left( n\right) =0.
\end{equation*}%
Moreover, since%
\begin{equation*}
\lim_{\alpha \downarrow 0}\int_{0}^{\alpha }\sqrt{\log \gamma _{\mathcal{H}%
}\left( \tau \right) }\text{d}\tau =\lim_{\alpha \downarrow
0}\int_{0}^{\alpha }\sqrt{\log \gamma _{\mathcal{G}}\left( \tau \right) }%
\text{d}\tau =0,
\end{equation*}%
then%
\begin{equation*}
\lim_{\alpha \downarrow 0}\underset{n\rightarrow +\infty }{\lim \sup }%
C\left( n,\alpha \right) =0,
\end{equation*}%
completing the proof of (\ref{equicontinuity}).

\bigskip

\textbf{Acknowledgements. }I would like to express my deepest gratitude to
Michel Broniatowski for carefully co-supervising my Ph.D. Thesis of which
this paper is part. I am also indebted to Giovanni Peccati for many
inspiring discussions that led me to the proof of Theorem 5.

\end{document}